%% file: article.tex
\documentclass[onefignum,onetabnum]{siamart171218}

\input{shared}

\ifpdf
\hypersetup{
  pdftitle={Homogenization of a Random Walk on a Graph in $\mathbb{R}^d$: An
    approach to predict macroscale diffusivity in media with finescale
    obstructions and interactions},
  pdfauthor={P. Donovan, M. Rathinam}
}
\fi

\begin{document}

\maketitle

\begin{abstract}
We propose random walks on suitably defined graphs as a framework for
finescale modeling of particle motion in an obstructed environment where the
particle may have interactions with the obstructions and the mean path length
of the particle may not be negligible in comparison to the finescale. This
motivates our study of a periodic, directed, and weighted graph embedded in
${\mathbb R}^d$ and the scaling limit of the associated continuous-time random
walk $Z(t)$ on the graph's nodes, which jumps along the graph's edges with
jump rates given by the edge weights. We show that the scaled process
$\varepsilon^2 Z(t/\varepsilon^2)$ converges to a linear drift $\bar{U}t$ 
and that $\varepsilon (Z(t/\varepsilon^2)-\bar{U}t/\varepsilon^2)$ converges weakly to a
Brownian motion. The diffusivity of the limiting Brownian motion can be
computed by solving a set of linear algebra problems. As we allow for jump
rates to be irreversible, our framework allows for the modeling of very
general forms of interactions such as attraction, repulsion, and bonding. 
The case of interest to us is that of null drift $\bar{U}=0$ and we provide some sufficient conditions for null drift that include certain symmetries of the graph. We also provide a formal asymptotic derivation of the effective diffusivity in analogy with homogenization theory for PDEs. For the case of reversible jump rates, we derive an equivalent variational formulation. This derivation involves developing notions of gradient for functions on the graph's nodes, divergence for ${\mathbb R}^d$-valued functions on the graph's edges, and a divergence theorem.
\end{abstract}



\input{inputs/introduction}
\input{inputs/preliminaries}
\input{inputs/rigorous-derivation}

\input{inputs/solvability-conditions}
\input{inputs/formal-derivation}
\input{inputs/variational-formulation}

\input{inputs/numerical-experiments}
\input{inputs/appendix}



\section*{Acknowledgments}
We thank the two anonymous reviewers for their constructive comments 
which helped improve this manuscript.

\bibliographystyle{siamplain}
\bibliography{references}
\end{document}

%% file: shared.tex
\usepackage{lipsum}
\usepackage{amsfonts}
\usepackage{graphicx}
\usepackage{epstopdf}
\usepackage{algorithmic}
\ifpdf
  \DeclareGraphicsExtensions{.eps,.pdf,.png,.jpg}
\else
  \DeclareGraphicsExtensions{.eps}
\fi

\usepackage{mathtools}

\newcommand{\real}{{\mathbb R}}

\newcommand\sS{\mathcal{S}}
\newcommand\sE{\mathcal{E}}
\newcommand\sL{\mathcal{L}}
\newcommand{\integ}{{\mathbb Z}}

\newcommand{\eps}{{\varepsilon}}
\newcommand{\E}{{\mathcal{E}}}
\newcommand{\Exp}{{\mathbb{E}}}
\newcommand\oS{\bar{\sS}}
\newcommand\oE{\bar{\sE}}

\newsiamremark{remark}{Remark}
\newsiamremark{hypothesis}{Hypothesis}
\crefname{hypothesis}{Hypothesis}{Hypotheses}
\newsiamthm{claim}{Claim}

\headers{Homogenization of a Random Walk on a Graph}{P. Donovan, M. Rathinam}

\title{Homogenization of a Random Walk on a Graph in $\real^{\lowercase{d}}$: An
    approach to predict macroscale diffusivity in media with finescale
    obstructions and interactions\thanks{Submitted to the editors DATE.
}}

\author{Preston Donovan\thanks{Department of Mathematics and Statistics Mathematics, University of Maryland Baltimore County, Baltimore, MD
  (\email{pdonovan@umbc.edu}, \email{muruhan@umbc.edu}).}
\and Muruhan Rathinam\footnotemark[2]}

\usepackage{amsopn}


%% file: inputs/introduction.tex
\section{Introduction}
We consider a periodic, weighted, directed graph embedded in the Euclidean
space $\real^d$ and the associated random walk $Z(t)$ in continuous time $t
\geq 0$. The random walk takes place on the nodes (vertices) of the graph and
jumps along the edges with jump rates given by the edge weights.
The scaled process $\eps^2 Z(t/\eps^2)$ converges to a
deterministic linear motion $\bar{U} t$, 
and the centered and rescaled process $\eps (Z(t/\eps^2)- \bar{U}t/\eps^2)$
converges weakly to a Brownian motion. We show that the diffusivity of the
limiting Brownian motion can be computed by solving a set of linear algebra
problems related to the graph structure.

The motivation for our problem arises from modeling the random
motion of a particle in an environment with obstructions. The particle may
have some form of interaction (attraction, repulsion, or bonding) with the
obstructions and the mean path length of the random motion of the particle may
be non-negligible compared to the characteristic finescale length of the
environment. Such a situation arises for a solute in an aqueous polymer gel
where the solute interaction with solvent is simply captured by a random walk
with certain mean path length, and the polymer molecular network is considered
as forming a stationary obstruction structure. Moreover, the solute may have interactions with the
polymer molecules in the
form of bonding, attraction, or repulsion. Since the ``pore sizes'' in the
gels can be as small as $10$ nanometers or less and the ``roughness'' of the pore
boundaries may only be a few nanometers, the mean path length in water
of a solute may be non-negligible. This leads one to consider, as a starting
point, a finescale model involving a random walk with nonzero path length or
jump sizes.

As an illustrative example,  we
consider a point particle undergoing a random motion in 2D where a quarter of the
region is obstructed with the obstructions being periodically placed squares.
We are interested in the situation when the finescale $\eps$ characterized by the
periodic spacing is very small compared to the domain of interest.
Typically, the
finescale description of the particle will be a random walk with non-zero
mean path length $\ell>0$ and hence the ratio $h=\ell/\eps$ of the path length to the periodic length
 will be a relevant factor. If $h \ll 1$ then the
finescale model may be taken to be Brownian motion
with reflections off the obstructed quarter, and hence may be
described by the diffusion equation with no-flux boundary conditions. Standard
homogenization theory for PDEs allows one
to compute the effective diffusivity of the macroscale limit.

\begin{figure}
\centering
	\includegraphics[width=60mm]{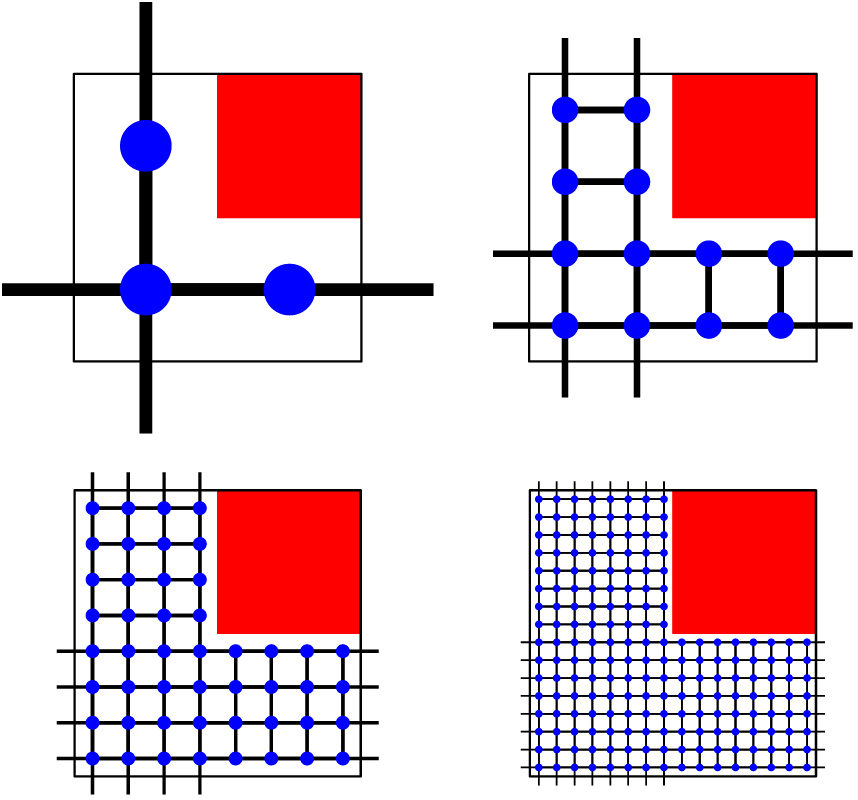}
	\caption{Periodic cells of graphs with nodes and edges defined in~\eqref{eqn:graph_sq_obs} for $h = 1/2,1/4,1/8$, and $1/16$. The circles represent edges, the thick lines represent edges, and the shaded region represents the ``obstructed'' region $\mathcal{O}$ where nodes are absent.}
\label{fig:graph-pd-cell}
\end{figure}

When $h$ is not negligible, the finescale model cannot be based on Brownian
motion. For instance, in the polymer gel example mentioned above, the solute
particle may be expected to move along linear 
trajectories which undergo sharp (non-differentiable) changes of direction due to collisions with the 
solvent molecules and or the polymer obstructions. Here the mean length of the
solute path between collisions with solvent molecules will be
captured by $h$. This suggests a finescale model that allows the particle to 
move in the region unoccupied by the stationary obstructions (polymer
molecules) with piecewise linear paths of i.i.d.\ length and direction. 
If a linear path encounters the boundary of the obstructed region, it 
will undergo specular reflection. Such a finescale model is continuous in 
space and even after homogenization one may expect to solve 
an infinite dimensional problem not unlike the unit-cell PDE 
one obtains in homogenization of PDEs.   

A computationally
tractable alternative to the continuous space model is 
to discretize the feasible spatial positions (nodes) and allow the
particle to jump (in continuous time) among those positions along certain ``edges''. The typical
spacing between the nodes of the edges will capture the mean path length $h$. See Figure
\ref{fig:graph-pd-cell} where the nodes are arranged in a Cartesian grid with
a maximum of four nearest neighbor edges where $h$ is the path length to
periodic length ratio with nodes being absent in the obstructed regions. If the jump rates are taken to be $1/h^2$, for a fixed
periodic length $\eps>0$ as $h \to 0$, the process limits to Brownian motion
with reflections. The theory developed in this paper
allows us to compute the limit of the process as $\eps \to 0$ for fixed $h$.
Figure \ref{fig:pathlengtheffects} drives home the message that the
effective diffusivity at the macroscale depends significantly on the ratio
$h$. Moreover, the $h \to 0$ limit of the macroscale diffusivity seems to
coincide with the homogenization of the Brownian motion with reflections.
If this commutativity (of discretization and homogenization) holds in general,
then the approach taken in this paper may also be used as a computational tool
to homogenize random motion described by Brownian motion with drift and
reflections since, after all, the unit-cell problem arising from the
homogenization of a PDE still needs to be solved by discretization.  However,
we do not pursue this question in this manuscript.

\begin{figure}
\centering
	\includegraphics[width=60mm]{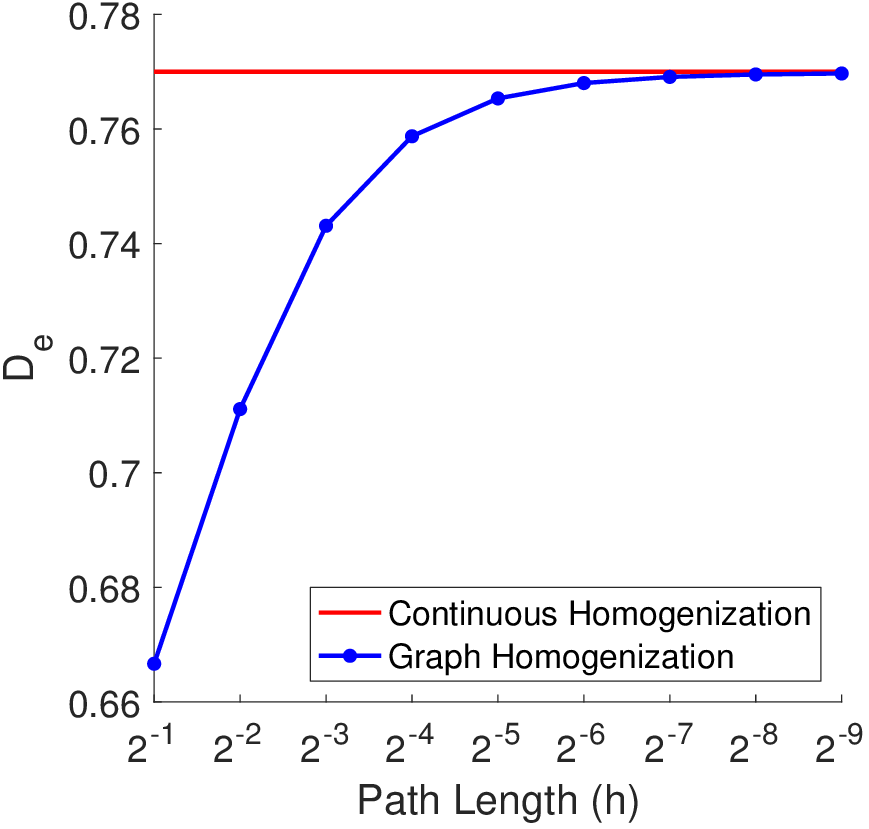}
	\caption{Effective diffusivity coefficients of the graphs depicted in Figure~\ref{fig:graph-pd-cell} and the effective diffusivity yielded by PDE homogenization. The effective diffusivity varies significantly for different path lengths ($h = 1/2, 1/4,\dots,1/512$) and appears to converge to the prediction of PDE homogenization.}
\label{fig:pathlengtheffects}
\end{figure}

In our example, the region of obstruction was fairly regular in the sense that
its boundary features (sidelength of the square) were not much smaller
than the periodic length. In many applications the obstructed region may be an elaborate shape with
boundary features that are of a scale much smaller than the periodic length,
and in these situations one needs to compare the mean path length with the
characteristic length of the boundary features. If the mean path length is not
much smaller in comparison, then a Brownian motion model (with drift, reflections
etc.\@) may not be a valid finescale model.

We also emphasize that in our approach, the spatial discretization can be
fairly arbitrary,
leading to fairly general graphs (subject to the periodicity assumption).
To the best of our knowledge, current literature on scaling limits of random
walks on graphs is limited to the special case of lattices $\integ^d$ where
edges connect the lattice's nodes with their $2d$ nearest
neighbors. Additionally, current literature assumes the jump rates to be reversible, meaning that the rate of jumping from $x$ to $y$ is equal to that of jumping from $y$ to $x$. We refer the reader to \cite{kunnemann1983diffusion, kipnis1986central, kozlov1987averaging} for the earliest results on this topic where the scaling limit of a random walk on a lattice with random conductances (assumed to be reversible) is considered. We note that the terminology ``conductance'' was used for jump rates. These works assumed nonzero conductances with a uniform lower bound. Later work \cite{berger2007quenched, faggionato2008random, mathieu2008quenched, biskup2011recent} studied the scaling limit of a random walk on a percolation lattice, where the conductances were allowed to be zero with a probability that ensures the existence of a unique infinite cluster. Additional related results can be found in \cite{owhadi2003approximation, gloria2011optimal, egloffe2014random, gloria2016quantitative}.

The reversible jump rates assumption does not allow one to capture fairly
general forms of interactions such as attraction, repulsion, and bond
formation mentioned above. In fact, the reversible jump rates assumption
implies that the generator of the Markov process describing the random
walk is symmetric
(self-adjoint). In the zero path length situation, that is, Brownian motion
with a drift vector field, self-adjointness does not hold for general
vector fields.  Our work does not assume reversible jump rates. Moreover, the graph structure considered is very general in that the set of nodes (vertices) is not confined to the lattice $\integ^d$ (or a subset of it) and the edges can be very general. On the other hand, while most of the aforementioned works consider either i.i.d.\ or stationary and ergodic jump rates, our work is restricted to the simpler periodic case. We believe that, by exploiting the ideas that connect the periodic case with the stationary and ergodic case \cite{owhadi2003approximation, gloria2016quantitative}, the results of this manuscript can be extended to the stationary and ergodic case.

Our approach to the convergence proofs uses the random time change
representation popularized by Kurtz and this framework has been
widely used to obtain several scaling limits of stochastic chemical reaction
networks \cite{ball2006asymptotic, kang2013separation, kang2014central}. We believe that exploitation of these ideas will enable future
results in homogenization of random walks on periodic graphs where the jump rates
exhibit multiple scales with jumps along some edges being slower compared to
the others.

From an application point of view, our previous work~\cite{donovan2016homogenization} demonstrated the feasibility of applying homogenization theory to predict solute motion in aqueous polymer solutions via comparison with experimental data. This work considered the simplest possible geometry of periodically placed spherical obstacles with the diffusion equation. To properly model polymer gels, more complex geometries and interactions with the solute are needed. We believe that the discrete graph approach provides a simple way to deal with complex obstruction geometries as well as interactions.

The paper is organized as follows. Details of the graph structure and its
induced random walk are outlined in Section~\ref{sect:preliminaries}.
In Section~\ref{sect:rigorous-derivation} we prove
that as $\eps \to 0$, $\eps^2 Z(t/\eps^2)$ converges to a deterministic
linear motion $\bar{U} t$ and $\eps^{-1} (\eps^2 Z(t/\eps^2) - \bar{U}t)$
converges
weakly to a Brownian motion whose diffusivity can be computed in terms of the
graph characteristics. We also show that certain types of modifications to $Z$
such as interpolation between jump times does not alter the above limiting
behavior.  
 Our main interest in applications is the situation
when $\bar{U}=0$. Section~\ref{sect:solvability} provides some
sufficient conditions under which $\bar{U}=0$.

In Section~\ref{sect:formal-derivation} we also provide a formal approach
to deriving the Brownian limit and the effective diffusivity via an asymptotic
expansion. This method mirrors the developments in PDE homogenization theory
and leads to the so-called ``unit-cell'' problem.
In PDE homogenization theory, when the operator is self-adjoint, the unit-cell
problem may also be formulated as
a variational problem where an energy functional is minimized.
Section~\ref{sect:variational-formulation} derives an equivalent variational
formulation that holds when the jump rates are reversible. This result
generalizes the existing results in the literature to graphs with nodes that
are not necessarily subsets of the integer lattice $\integ^d$.
We develop a notion of {\em gradient} and {\em divergence} for graph functions
and prove a {\em divergence theorem} in order to derive our variational
formulation of the unit-cell problem. We note that a reader uninterested
in connections with the PDE homogenization theory may skip Sections
\ref{sect:formal-derivation} and \ref{sect:variational-formulation}.

In Section~\ref{sect:numerics} we use the same planar obstruction 
geometry shown in Figure~\ref{fig:graph-pd-cell} to numerically 
illustrate the effects of path length variation as well interactions of the random
walker with obstructions (attraction, repulsion, or bonding) on the
effective diffusivity. We also compare via Monte Carlo simulations the effects
of spatial discretization to obtain a discrete set of nodes as required by our
homogenization theory. Appendix A confirms that the effective diffusivity
predicted by the formal asymptotics approach is correct.

%% file: inputs/preliminaries.tex
\section{Mathematical preliminaries}\label{sect:preliminaries}
\subsection{Directed graph}
Define a \emph{weighted, directed graph} to be a triple $(\sS,\sE,\lambda)$,
where $\sS \subset \real^d$ is a countable set of \emph{nodes}, $\sE \subset
\sS \times \sS$ is the \emph{directed edge set}, and $\lambda:\sE \rightarrow
(0,\infty)$ is a positive edge function. We will use the term \emph{graph} and
\emph{weighted, directed graph} interchangeably. A directed graph is
\emph{strongly connected} if each node can reach every other node via
traversing edges in the direction they point. We note that our definition
does not allow more than one directed edge from a node $x$ to another node
$y$.

{\bf Assumption 1:} The graph $(\sS,\sE,\lambda)$ is strongly connected,
does not have self edges (edges of the form $(y,y)$), is invariant under integer translations, and the node set $\sS \subset \real^d$ is countable.

Invariance under integer vector translations means for all $n \in \integ^d$,
\begin{enumerate}
  \item If $x \in \sS$ then $x + n \in \sS$.
  \item If $(x,y) \in \sE$ then $(x + n, y + n) \in \sE$.
  \item If $(x,y) \in \sE$ then $\lambda(x+n,y+n) = \lambda(x,y)$.
\end{enumerate}
We sometimes write $\lambda(e)$ as $\lambda_e$. For an edge $e=(x,y) \in \sE$, denote the \emph{originating node} $\partial_-e = x$ and the \emph{terminal node} $\partial_+e = y$. Given a node $y \in \sS$, let $\sE_y$ denote the subset of edges that originate in $y$ and conversely let $\sE'_y$ denote the subset of edges that terminate in $y$. That is,
\begin{equation*}
\sE_y = \{e \in \sE : \partial_-e = y\} \text{ and }
\sE'_y = \{e \in \sE : \partial_+e = y\}.
\end{equation*}

\subsection{Random walk on the graph}
A graph $(\sS,\sE,\lambda)$ satisfying Assumption 1 induces a Markov process
$Z(t)$ in continuous time taking values in $\sS$ with \emph{intensity} or
\emph{jump rate} from $x \in \sS$ to $y \in \sS$ given by $\lambda(x,y)$. 
This means that, conditioned on $Z(t)=x$, the probability that the process
jumps from $x$ to $y$ during $(t,t+h]$  
is $\lambda(x,y) h + o(h)$ as $h \to 0+$.  It also follows from Markov process 
theory that the probability that more than one jump occurs during $(0,h]$ is $o(h)$
as $h \to 0+$. 
The \emph{generator} $\mathcal{L}_Z$ of $Z$ is an operator on a subset 
of $\real^{\sS}$ which may be regarded as the conditional 
rate of change
\[
(\mathcal{L}_Z f)(y) = \lim_{h \to 0+} \Exp(f(Z(t+h)) - f(Z(t)) \, | \, Z(t)=y)/h, 
\]
for those bounded functions $f \in \real^{\sS}$ for which this limit exists. 
The generator of $Z$ is thus given by
\begin{equation*}\label{eq_genZ}
(\mathcal{L}_Z f)(y) = \sum_{(y,z) \in \sE_y} (f(z) - f(y)) \lambda(y,z).
\end{equation*}
If $\sE_y$ is finite for all $y$, this is defined for all
bounded functions $f \in \real^{\sS}$. See \cite{bremaud2013markov} for a
reference. 

Introduce the following equivalence relation on $\sS$: two nodes $x,y \in \sS$ are equivalent if and only if $x-y \in \integ^d$. Let $\oS$ denote the set of equivalent classes of $\sS$. We shall identify $\oS= \sS \cap [0,1)^d$. By Assumption 1, $\oS$ is strongly connected. Let $\Pi:\sS \rightarrow \oS$ be the natural projection that maps an element to its equivalent class.

Introduce the following equivalence relation on $\sE$: two edges
$(x,x'),(y,y') \in \sE$ are equivalent if and only if there exists $n \in
\integ^d$ such that $x-y=x'-y'=n$. Let $\oE$ denote the set of equivalence
classes of $\sE$. Note that $\oE$ may not be regarded as a subset of $\oS
\times \oS$ and that we may identify $\oE = \cup_{y \in \oS} \sE_y$. By
Assumption 1, if $e_1,e_2 \in \sE$ belong to the same equivalence class then
$\lambda(e_1) = \lambda(e_2)$.

{\bf Assumption 2:} The node set $\oS$ and edge set $\oE$ contain finitely many elements.

Thus, $\sS$ and $\sE$ are countable.
For any edge function $g\in (\real^d)^{\sE}$ that is periodic, that is, satisfies $g(e+(n,n))=g(e)$
for all $e \in \sE$ and $n \in \real^d$,  we have the following useful identities:
\begin{equation}\label{eqn:sum-of-sums}
\sum_{e \in \oE} g(e) = \sum_{y \in \oS} \sum_{e \in \sE_y} g(e) = \sum_{y \in \oS} \sum_{e \in \sE'_y} g(e).
\end{equation}
By Assumption 2, these summations are well defined.

By Assumptions 1 and 2, the projected process
\begin{equation*}
Y(t) = \Pi Z(t)
\end{equation*}
is a Markov process with the finite state space $\oS$ that jumps along edges
in $\sE_{\Pi} \subset \oS \times \oS$ where $(y,z) \in \sE_{\Pi}$ if and
only if there exists $(y',z') \in \sE$ such that $\Pi(y')=y$ and $\Pi(z')=z$.
The jump rate is given by $\bar{\lambda}: \sE_{\Pi} \rightarrow (0,\infty)$ where
\begin{equation*}
\bar{\lambda}(x,y) = \sum_{\substack{(x,z) \in \sE_x\\ \pi(z) = y}} \lambda(x,z)
\end{equation*}
for $(x,y) \in \sE_{\Pi}$.
The relationship between $\oE$ and $\sE_{\Pi}$ is important to note.
Let the map $\kappa:\oE \to \sE_\Pi$ be defined as follows.
Given $e \in \oE$, let $(y',z') \in \sE$ be any edge in the equivalence class
$e$. Set $\kappa(e) = (\Pi(y'),\Pi(z')) \in \sE_\Pi$. This is well defined,
and is a surjection. For most of the applications we are
interested in, $\kappa$ is a bijection, but we do not need this assumption
for most of the results in this paper.


For each edge $e=(x,y) \in \sE$, we denote the \emph{jump size} by $\nu_e = y-x$. Hence, using the random time change representation~\cite{ethier2005markov} we can write
\begin{equation}\label{eqn:z}
Z(t) = z_0 + \sum_{e \in \oE}  \nu_e R_e\left( \lambda_e \int_0^t 1_{\{\partial_-e\}}(Y(s)) ds \right),
\end{equation}
where $\{R_e\}_{e \in \oE}$ is a collection of independent unit-rate Poisson
processes and we have used the identification $\oE = \cup_{y \in \oS} \sE_y$.
We shall take $z_0=0$ without loss of generality throughout the
rest of this paper.

Define the \emph{rate matrix} $L:\oS \times \oS \rightarrow \real$ of $(\sS,\sE,\lambda)$ by
\begin{equation}\label{eqn:rate-matrix}
\begin{aligned}
L(x,y) =
\begin{cases}
      \bar{\lambda}(x,y) & x \neq y\\
      -\sum_{y \in \oS \backslash \{x\}} L(x,y) & x = y \\
		0 				& \text{otherwise.}
\end{cases}
\end{aligned}
\end{equation}
We may regard $L$ as a linear map $\real^{\oS} \rightarrow \real^{\oS}$ where
\begin{equation}\label{eqn:L-op}
(Lf)(y) = \sum_{e \in \sE_y} (f(y+\nu_e)-f(y)) \lambda_e,
\end{equation}
for $f \in \real^{\oS}$. In this view, $L$ is the generator of the process $Y$ and we note that $y+\nu_e$ is calculated modulo $1$.

For $y \in \oS$ define
\begin{equation}
\lambda^0(y) = \sum_{e \in \sE_y} \lambda_e.
\end{equation}
The transpose of $L$, given by $L^T(y,z) = L(z,y)$, can also be regarded as a linear operator on $\real^{\oS}$ where
\begin{equation}\label{eqn:Lt-op}
(L^Tf)(y) = \sum_{e \in \sE_y'} f(y-\nu_e) \lambda_e - f(y) \lambda^0(y),
\end{equation}
for $f \in \real^{\oS}$, and as before $y-\nu_e$ is computed modulo $1$. Throughout this paper we typically regard $L$ and $L^T$ as linear operators rather than functions on $\oS \times \oS$. The appropriate interpretation will be obvious.

Standard Markov process theory shows that $Y$ is an ergodic process because
$\oS$ is finite and $(\oS,\sE_{\Pi})$ is strongly connected (which follows
from Assumption 1). Hence, $L^T$ has
a one-dimensional null space that contains a function $\pi \in \real^{\oS}$
such that $\pi(y) > 0$ for all $y \in \oS$ and $\sum_{y \in \oS} \pi(y)
=1$. In other words, $\pi$ is the unique stationary probability measure of
process $Y$.

\subsection{A law of large number and central limit argument}
Here we briefly sketch a line of reasoning that makes clear why the suitably scaled process
has a deterministic limit via law of large numbers and a central limit
correction. Taking $z_0=0 \in \sS$ as the initial node and $y_0=0 \in \oS$ being its
projection, let $T_0=0$ and denote by $T_n$ for $n=1,2,\dots$ the successive return times
to $y_0=0$ of the projected process $Y$. Due to the Markov property,
$(T_n-T_{n-1},Z(T_n)-Z(T_{n-1}))$ for $n=1,2,\dots$ form an i.i.d.\ sequence.
Thus, the law of large number and the central limit results are applicable
to this sequence provided appropriate integrability. Suppose that the expected value of the increment of
the process $Z$ between successive revisits of $Y$ is zero,
that is, $\mathbb{E}(Z(T_n)-Z(T_{n-1}))=0$ or alternatively we centralize $Z$ by subtracting the
drift. Then, in light of Theorem 14.4 from
\cite{billingsley2011convergence}, one could expect that the centralized and
suitably scaled $Z$ converges to
a $d$-dimensional Brownian motion in $D^d[0,\infty)$ if certain integrability
  conditions can be verified. However, we are not
merely interested in knowing that the centralized process converges to
a Brownian motion. We want to find a way to compute the diffusivity matrix
of the Brownian limit. In order to do this, one needs to
compute the covariance of the increment $Z(T_n)-Z(T_{n-1})$. In the graph
setting, this involves considering the possible paths of $Z$ that result
in revisits of $Y$ to the original state. Direct determination of this
seems a harder task than the approach we take in this paper.

Our approach to finding the limits involves the use of random time change
representation, functional law of large numbers for Poisson processes, and
the martingale central limit theorem. At first
glance, our approach may appear as using a powerful weapon
on a simple i.i.d.\ sum type problem. However, the computations involved are simpler
and do not involve enumerating all possible return paths as required by
the direct approach. Moreover, the framework of using random time change
representation will allow us, in future work, to consider separation of time scales
in the jump process.

%% file: inputs/rigorous-derivation.tex
\section{Scaling limits of the random walk}\label{sect:rigorous-derivation}
Consider a weighted, directed graph \\ $(\sS,\sE,\lambda)$ satisfying assumptions in Section~\ref{sect:preliminaries}. Recall the random time change representation of $Z$ (\ref{eqn:z}):
\begin{equation*}
Z(t) = z_0 + \sum_{e \in \oE} \nu_e R_e\left( \lambda_e \int_0^t 1_{\{\partial_-e\}}(Y(s)) ds \right).
\end{equation*}
We assume that the processes $Z, R_e$, and hence $Y$ are all
carried by a probability space $(\Omega,\mathcal{F},\mathbb{P})$. 
We also suppose that $R_e$ and hence $Z$ and $Y$ are {\em cadlag} (that is have right continuous 
paths with left hand limits).  
For the remainder of this discussion, we assume without loss of generality
$z_0 = 0$.

It is instructive to define the \emph{drift field} $\rho:\oS \to \real^d$ by
\begin{equation}\label{eqn:drift-field}
\rho(y) = \sum_{e \in \sE_y} \nu_e \lambda_e.
\end{equation}
Note that $\rho(y)$ is the ``instantaneous drift rate'':
\begin{equation}
\lim_{t \to 0+} \frac{1}{t} \mathbb{E}(Z(t_0 + t)-z \, | \, Z(t_0)=z)
= \sum_{e \in \sE_z} \lim_{t \to 0+} \frac{\nu_e \lambda_e t + o(t)}{t} =
\sum_{e \in \sE_z} \nu_e \lambda_e = \rho(\Pi(z)).
\end{equation}
Intuitively, we expect that the scaled process $\eps^2 Z(t/\eps^2)$
converges (as $\eps \to 0$) almost surely to the
deterministic linear motion given by
\[
t \sum_{y \in \oS} \rho(y) \pi(y) = t \sum_{e \in \oE} \nu_e \lambda_e
\pi(\partial_- e),
\]
because, in the long run, the fraction of time spent by the projected process
$Y$ on $y \in \oS$ is equal to $\pi(y)$. This limit is shown in Section
\ref{sec-drift-lim} using the functional law of large numbers for Poisson
processes and the ergodicity of $Y$. Thus, we shall refer to $\bar{U}$ defined by
\begin{equation}\label{eq-Ubar}
\bar{U} = \sum_{y \in \oS} \rho(y) \pi(y)
\end{equation}
as the {\em long run drift rate}.

For convenience we define the centered Poisson processes
\begin{equation}
	\tilde{R}_e(t) = R_e(t) - t.
\end{equation}

We also define processes $S^\eps_e$ and $M^\eps_e$ as follows:
\begin{equation}\label{eqn:S_e}
S_e^{\eps}(t) = \lambda_e \int_0^{t/\eps^2} 1_{\{\partial_-e\}}(Y(s)) ds
\end{equation}
and
\begin{equation}\label{eqn:M_e}
M^{\eps}_e(t) = \tilde{R}_e(S_e^{\eps}(t)).
\end{equation}

\subsection{Pure drift limit of $\eps^2 Z(t/\eps^2)$}\label{sec-drift-lim}
Consider the process $U_\eps$ defined by $U_\eps(t) = \eps^2
Z(t/\eps^2)$. We may write
\begin{equation} \label{eqn:W}
U_{\eps}(t) = \eps^2 \sum_{e \in \oE} \nu_e M^{\eps}_e(t) + \eps^2\sum_{e \in \oE}\nu_e S_e^{\eps}(t).
\end{equation}
We show that as $\eps \rightarrow 0$, $U_{\eps}(t) \rightarrow t \bar{U}$ almost surely and uniformly on compact subintervals
$[t_0,T]$ of time that exclude $0$.

\begin{lemma} \label{lem:epsM_lim}
For all $T > 0$,
\begin{equation*}
\lim_{\eps \rightarrow 0}\sup_{t \in [0,T]} \eps^2 |M^{\eps}_e(t)| = 0.
\end{equation*}
\end{lemma}
\begin{proof}
Notice that $0 \leq \eps^2 S_e^{\eps}(t) \leq \lambda_e t$  for all $t$. Then
\begin{equation*}
\begin{aligned}
\sup_{t \in [0,T]} |\eps^2 M^{\eps}_e(t)|
	= \sup_{t \in [0,T]} |\eps^2 \tilde{R}_e(S_e^{\eps}(t))|
	= \sup_{t \in [0,T]} |\eps^2 R_e(S_e^{\eps}(t)) - \eps^2 S_e^{\eps}(t)| \\
	= \sup_{u \in [0,\eps^2 S_e^{\eps}(T)]} |\eps^2 R_e\Big(\frac{u}{\eps^2}\Big) - u|
	\leq \sup_{u \in [0,\lambda_e T]} |\eps^2 R_e\Big(\frac{u}{\eps^2}\Big) - u|
	\rightarrow 0 \text{ a.s.\ as } \eps \rightarrow 0,
\end{aligned}
\end{equation*}
where the limit follows from the functional law of large numbers for Poisson processes~\cite{ethier2005markov}.
\end{proof}
\begin{lemma}\label{lem:w-eps}
For fixed $0 < t_0 <T <\infty$,
\begin{equation*}
\lim_{\eps \rightarrow 0} \sup_{t \in [t_0,T]} |U_{\eps}(t) - t \bar{U}| =0.
\end{equation*}
\end{lemma}
\begin{proof}
In light of \eqref{eq-Ubar} and Lemma~\ref{lem:epsM_lim} it is adequate to show that
\[
\lim_{\eps \to 0} \sup_{t \in [t_0,T]} |\eps^2 S^\eps_e(t) - t \lambda_e
\pi(\partial_- e)| =0.
\]
First we obtain the (almost sure) upper bound
\[
\sup_{t \in [t_0,T]} |\eps^2 S_e^{\eps}(t) - t \lambda_e \pi(\partial_- e)|
\leq T \lambda_e \left| \frac{1}{t/\eps^2} \int_0^{t/\eps^2}
1_{\{\partial_-e\}}(Y(s)) ds - \pi(\partial_- e)\right|.
\]
By the ergodicity of $Y$, given $\delta>0$ there exists $T_\delta>0$ such that
for all $\tilde{T} \geq T_\delta$,
\[
\left| \frac{1}{\tilde{T}} \int_0^{\tilde{T}}
1_{\{\partial_-e\}}(Y(s)) ds - \pi(\partial_- e)\right|<\delta.
\]
Then for all $0<\eps<\sqrt{t_0/T_\delta}$  we see that
\[
\sup_{t \in [t_0,T]} \left| \frac{1}{t/\eps^2} \int_0^{t/\eps^2}
1_{\{\partial_-e\}}(Y(s)) ds - \pi(\partial_- e)\right| < \delta.
\]
The result follows from this.
\end{proof}
The following important lemma is repeatedly used in the next subsection.
\begin{lemma}\label{lem-ReSe-limit}
For each $e \in \oE$ and $t \geq 0$
\[
\lim_{\eps \to 0} \eps^2 R_e(S^\eps_e(t)) = t \lambda_e \pi(\partial_- e).
\]
\end{lemma}
\begin{proof} The proof follows from Lemmas \ref{lem:epsM_lim} and \ref{lem:w-eps}.
\end{proof} 

\subsection{Limiting behavior of the centered and rescaled process $Z_{\eps}$}
The most interesting scenario for us is when the long run drift rate $\bar{U}$ is
zero. We shall refer to this as the {\em null drift condition}:
\begin{equation}\label{eq:null-drift}
\bar{U} = \sum_{y \in \oS} \rho(y) \pi(y) = \sum_{e \in \oE} \nu_e \lambda_e
\pi(\partial_- e) =0.
\end{equation}
Nevertheless, we shall consider the general case $\bar{U} \neq 0$, 
 and define the centered and rescaled process $Z_\eps(t)$
\begin{equation}\label{eqn:Zeps}
Z_\eps(t) = \eps^{-1} (U_\eps(t) - t \bar{U}) = \eps(Z(t/\eps^2)-t \bar{U}/\eps^2).
\end{equation}

We define the {\em centered drift field} $\tilde{\rho}$ by
\begin{equation}\label{eqn:tilderho}
\tilde{\rho}(y) = \rho(y) - \bar{U}.
\end{equation} 
We also define $\psi \in (\real^d)^{\oS}$ by
\begin{equation}\label{eqn:Lg_fd}
L\psi = \tilde{\rho}.
\end{equation}
The following observation 
ensures that such a $\psi$ exists and we note that $\psi$ is unique only up to an
additive constant. 
Recall that $L^T$ has a one-dimensional null space by ergodicity of $Y$. Hence
we have the following equivalent statements:
\begin{equation*}
\tilde{\rho} \in R(L) \iff \tilde{\rho} \perp N(L^T)
\iff \sum_{y \in \oS} \tilde{\rho}(y) \pi(y) = 0 \iff \sum_{y \in \oS} \rho(y) \pi(y) = \bar{U}.
\end{equation*}

Now we state our main result.
\begin{theorem}\label{thm:deff_rigorous}
The process $Z_{\eps}$ defined by \eqref{eqn:Zeps} converges weakly
\begin{equation*}
Z_{\eps} \Rightarrow {\bf Z} \text{ in } D^{d}[0,\infty) \text{ as } \eps \rightarrow 0,
\end{equation*}
where ${\bf Z}$ is a Brownian motion with
\begin{equation*}
\begin{aligned}
\Exp {\bf Z}(t) &= 0, \\
\Exp [{\bf Z}(t){\bf Z}(t)^T] &= 2 C t.
\end{aligned}
\end{equation*}
Here,
\begin{equation} \label{eqn:C}
\begin{aligned}
C = \frac{1}{2} \sum_{e \in \oE} \alpha_e \alpha_e^T \lambda_e \pi(\partial_-e)
\end{aligned}
\end{equation}
where
\begin{equation}\label{eqn:alpha}
\alpha_e = \nu_e - (\psi(\partial_+e) - \psi(\partial_-e)),
\end{equation}
and $\psi$ satisfies~(\ref{eqn:Lg_fd}).
\end{theorem}

The proof of Theorem \ref{thm:deff_rigorous} will be be presented after we establish some lemmas. 

An immediate consequence of Theorem \ref{thm:deff_rigorous} is the following
lemma. 
\begin{lemma}
The diffusivity matrix $C$~(\ref{eqn:C}), which is symmetric positive semi-definite,
is symmetric positive definite if and only if $\{\alpha_e\}_{e\in \oE}$~(\ref{eqn:alpha}) spans $\real^d$ .
\end{lemma}
\begin{proof}
The result is clear since for $x \in \real^d$,
\begin{equation*}
x^T C x = \sum_{e \in \oE} x^T \alpha_e \alpha_e^T x\lambda_e \pi(\partial_-e) = \sum_{e \in \oE} |x^T \alpha_e|^2 \lambda_e \pi(\partial_-e).
\end{equation*}
\end{proof}

Our approach for showing the limit of $Z_\eps$ is to use
the martingale functional central limit theorem \cite{ethier2005markov}.
The key idea is related to
\cite{bhattacharya1982functional}, but we follow the ideas articulated concisely in \cite{kang2014central}.
The first step is to write $Z_\eps$ as a sum of a martingale and a term
that vanishes as $\eps \to 0+$.

Define the filtration $\{\mathcal{F}^\eps_t\}_t$ by
\[
\mathcal{F}^\eps_t = \sigma(Z_\eps(s),R_e(S^\eps_e(s)) \, ; \, 0 \leq s \leq t, e \in
\oE).
\]
Since $S^\eps_e(t) \leq \lambda_e t$, the centered counting process
$M^\eps_e(t)=\tilde{R}_e(S^\eps_e(t))$ is a martingale with respect to
$\{\mathcal{F}^\eps_t\}_t$.

It follows from \eqref{eqn:Zeps} that
\begin{equation}
\begin{aligned}
Z_{\eps}(t) &= \eps \sum_{e \in \oE} \nu_e R_e\left( \lambda_e \int_0^{t/\eps^2} 1_{\{\partial_-e\}}(Y(s)) ds \right)
	 + \frac{t \bar{U}}{\eps}\\
	 &= \eps \sum_{e \in \oE} \nu_e M^{\eps}_e(t) + \eps \sum_{e \in
           \oE}\nu_e S_e^{\eps}(t) + \frac{t \bar{U}}{\eps}\\
&= \eps \sum_{e \in \oE} \nu_e M^{\eps}_e(t) + \eps \int_0^{t/\eps^2}
         \rho(Y(s)) ds + \frac{t \bar{U}}{\eps}.
\end{aligned}
\end{equation}

We note that using the definition of $\tilde{\rho}$, we may write
\begin{equation}\label{eqn:Z_eps}
Z_\eps(t) = \eps \sum_{e \in \oE} \nu_e M^\eps_e(t) + \eps \int_0^{t/\eps^2} \tilde{\rho}(Y(s)) ds.
\end{equation}
Since the first sum is a martingale, we need to relate the second term
to a martingale.
To that end, we first define the Dynkin's martingale
\begin{equation}\label{eqn:N}
N(t) = \psi(Y(t)) - \psi(Y(0)) - \int_0^t (L\psi)(Y(s)) ds.
\end{equation}
Using the relationship $L \psi = \tilde{\rho}$ we may write
\begin{equation}\label{eqn:Z_eps2}
Z_{\eps}(t) = \sum_{e \in \oE} \eps \nu_e M^\eps_e(t) - \eps N(t/\eps^2)
 + \eps[\psi(Y(t/\eps^2)) - \psi(Y(0))].
\end{equation}
Since $\psi$ is a bounded function
\begin{equation}
\lim_{\eps \rightarrow 0}	\eps[\psi(Y(t/\eps^2)) - \psi(Y(0))] = 0.
\end{equation}
In order to apply the martingale central limit theorem, we compute the
quadratic variation of the martingale
\[
\sum_{e \in \oE} \eps \nu_e M^\eps_e(t) - \eps N(t/\eps^2),
\]
which can be expanded as follows:
\begin{equation}\label{eqn:Ze_terms}
\begin{aligned}
&\Big[\sum_{e \in \oE} \eps \nu_e M^{\eps}_e(\cdot) - \eps N(\cdot/\eps^2)\Big](t) \\
	=&\; [\sum_{e \in \oE} \eps \nu_e M^{\eps}_e(\cdot),\sum_{e \in \oE} \eps \nu_e M^{\eps}_e(\cdot)](t)
	- [\sum_{e \in \oE} \eps \nu_e M^{\eps}_e(\cdot),\eps N(\cdot/\eps^2)](t) \\
	&- [\eps N(\cdot/\eps^2),\sum_{e \in \oE} \eps \nu_e M^{\eps}_e(\cdot)](t)
	+ [\eps N(\cdot/\eps^2),\eps N(\cdot/\eps^2)](t).
\end{aligned}
\end{equation}
Before demonstrating convergence of (\ref{eqn:Ze_terms}), we note the
following basic result.
\begin{lemma}[See \cite{klebaner2012introduction} for instance.]\label{lem:covariation}
Let $f,g : \real \rightarrow \real$. If $f$ is continuous and $g$ is of finite variation, then their covariation is zero: $[f,g](t) = 0$.
\end{lemma}
The following three lemmas summarize the convergence of the right-hand side of
(\ref{eqn:Ze_terms}). The key points to note are that the processes
$M^\eps_e$ and $N(t/\eps^2)$ are of finite variation, $S^\eps_e(t)$
is absolutely continuous in $t$, and, moreover, $R_a(S^\eps_a(\cdot))$
and $R_b(S^\eps_b(\cdot))$ for $a \neq b$ do not have common jumps (almost
surely). In particular, when computing the quadratic covariations, only the
(common) jumps matter.
\begin{lemma}\label{lem:quadcov_term1}
For every $t \geq 0$, the quadratic variation converges almost surely:
\begin{equation*}
\lim_{\eps \rightarrow 0} [\eps \sum_{e \in \oE} \nu_e M^{\eps}_e(\cdot),\eps \sum_{e \in \oE} \nu_e M_{e}(\cdot)](t)
	= t \sum_{e \in \oE} \lambda_e \nu_e \nu_e^T \pi(\partial_-e).
\end{equation*}
\end{lemma}
\begin{proof}
For $a,b \in \oE$ we have
\begin{equation*}
[\eps M^{\eps}_a,\eps M^{\eps}_b](t)
	= \eps^2\Big[R_a(S_a^{\eps}(\cdot)), R_b(S_b^{\eps}(\cdot))\Big](t).
\end{equation*}
If $a \neq b$ this term is zero. If $a=b$ we get
\[
[\eps M^{\eps}_a,\eps M^{\eps}_a](t) = \eps^2 R_a(S_a^{\eps}(t)),
\]
which converges to $t \lambda_a \pi(\partial_-a)$ by Lemma \ref{lem-ReSe-limit}.
\end{proof}

\begin{lemma}\label{lem:quadcov_term2_3}
Fix $t \geq 0$. Let $\psi$ satisfy~(\ref{eqn:Lg_fd}). Then the quadratic covariation converges almost surely:
\begin{equation*}
\begin{aligned}
\lim_{\eps \rightarrow 0} [\sum_{e \in \oE} \eps \nu_e M^{\eps}_e(\cdot),\eps N(\cdot/\eps^2)](t)
&= t \sum_{e \in \oE} \nu_e \big(\psi(\partial_+e) - \psi(\partial_-e)\big)^T
\lambda_e \pi(\partial_-e)\\
\lim_{\eps \rightarrow 0} [\eps N(\cdot/\eps^2), \sum_{e \in \oE} \eps \nu_e M^{\eps}_e(\cdot),](t),
&= t \sum_{e \in \oE} \big(\psi(\partial_+e) - \psi(\partial_-e)\big) \nu_e^T
\lambda_e \pi(\partial_-e).
\end{aligned}
\end{equation*}
\end{lemma}
\begin{proof}
We show the first limit. Substituting the definition of $N$ we obtain
\begin{equation*}
\Big[\sum_{e \in \oE} \eps \nu_e M^{\eps}_e(\cdot),\eps
  N(\cdot/\eps^2)\Big](t) = \sum_{e \in \oE} \eps^2 \Big[\nu_e
  R_e(S_e^{\eps}(\cdot)),\psi(Y(\cdot/\eps^2))\Big](t).
\end{equation*}
Now
\begin{equation*}
\Big[\nu_e R_e(S_e^{\eps}(\cdot)),\psi(Y(\cdot/\eps^2))\Big](t) = \nu_e
(\psi(\partial_+ e)-\psi(\partial_-e))^T R_e(S_e^{\eps}(t)).
\end{equation*}
As before, Lemma \ref{lem-ReSe-limit} yields the result.
\end{proof}
\begin{lemma}\label{lem:quadcov_term4}
Define $N$ as in~(\ref{eqn:N}).  Let $\psi$ satisfy~(\ref{eqn:Lg_fd}). Then the quadratic variation converges almost surely:
\begin{equation*}
\lim_{\eps \rightarrow 0} [\eps N(\cdot/\eps^2),\eps N(\cdot/\eps^2)](t)
= t\sum_{e \in \oE} (\psi(\partial_+e)-\psi(\partial_-e))(\psi(\partial_+e)-\psi(\partial_-e))^T \lambda_e\pi(\partial_-e).
\end{equation*}
\end{lemma}
\begin{proof}
We have
\begin{equation*}
\begin{aligned}
[\eps N(\cdot/\eps^2), \eps N(\cdot/\eps^2)](t) &= \eps^2
[\psi(Y(\cdot/\eps^2)), \psi(Y(\cdot/\eps^2))](t)\\
&= \eps^2 \sum_{e \in \oE} (\psi(\partial_+e)-\psi(\partial_-e))(\psi(\partial_+e)-\psi(\partial_-e))^T  R_e(S^{\eps}_e(t)).
\end{aligned}
\end{equation*}
As before, the result follows from Lemma \ref{lem-ReSe-limit}.
\end{proof}

\begin{proof}(of Theorem \ref{thm:deff_rigorous})
We first show that the maximum jump of 
$$\sum_{e \in \oE} \eps \nu_e M^{\eps}_e(t) - \eps N(t/\eps^2)$$ is asymptotically negligible. Note that $M^{\eps}_e$ has a maximum jump size of 1 because $R_e$ is a Poisson process and $S_e^{\eps}$ is continuous. Also, the jumps of $N(t/\eps^2)$ are bounded above by some constant $k$ because $\psi$ is bounded. Thus,
\begin{equation*}
\begin{aligned}
&\lim_{\eps \rightarrow 0} \Exp\bigg[\sup_{s \leq t}\Big|\sum_{e \in \oE} \eps \nu_e M^{\eps}_e(s) - \eps N(s/\eps^2)
		-\Big(\sum_{e \in \oE} \eps \nu_e M^{\eps}_e({s-}) - \eps N({s-}/\eps^2)\Big)\Big|\bigg] \\
&= \lim_{\eps \rightarrow 0} \eps\Exp\bigg[\sup_{s \leq t}\Big|\sum_{e \in \oE} \big(\nu_e M^{\eps}_e(s) - \nu_e M^{\eps}_e({s-})\big)
		- \big(N(s/\eps^2)- N({s-}/\eps^2\big)\Big|\bigg] \\
&\leq \lim_{\eps \rightarrow 0} \eps\Exp\bigg[\sum_{e \in \oE} |\nu_e| + k\bigg]
 = 0.
\end{aligned}
\end{equation*}
Next, applying Lemmas~\ref{lem:quadcov_term1},~\ref{lem:quadcov_term2_3}, and~\ref{lem:quadcov_term4}, the limit of the quadratic covariation is
\begin{equation*}
\lim_{\eps \rightarrow 0}\Big[\sum_{e \in \oE} \eps \nu_e M^{\eps}_e(\cdot) -
  \eps N(\cdot/\eps^2)\Big](t) =
 t\sum_{e \in \oE} \alpha_e \alpha_e^T \lambda_e \pi(\partial_-e).
\end{equation*}
Hence we can apply the martingale functional central limit theorem from~\cite{ethier2005markov} to the martingale $\sum_{e \in \oE} \eps \nu_e M^{\eps}_e(t) - \eps N(t/\eps^2)$ to get the following weak convergence:
\begin{equation*}
\sum_{e \in \oE} \eps \nu_e M^{\eps}_e(\cdot) - \eps N(\cdot/\eps^2) \Rightarrow {\bf Z} \text{ in } D^{d}[0,\infty) \text{ as } \eps \rightarrow 0.
\end{equation*}
Recall that
\begin{equation*}
Z_{\eps}(t)
= \sum_{e \in \oE} \eps \nu_e M^{\eps}_e(t) - \eps N(t/\eps^2) + \eps[\psi(Y(t/\eps^2)) - \psi(Y(0))].
\end{equation*}
Because the first two terms converge weakly to a Brownian motion ${\bf Z}$,
and the terms involving $\psi$ converge a.s.\ to 0, the entire process
$Z_{\eps}$ converges weakly to ${\bf Z}$. See Theorem 3.1 of \cite{billingsley2011convergence} for a proof of this result.
\end{proof}

\subsection{Modified process $I$}\label{sec:processI}
In this subsection we consider modifications to the random walk $Z$ 
for obtaining better models of the finescale. 
We note that the process $Z$ is confined to the state space $\sS$ and 
is a jump process. A natural modeling question is what happens if a more realistic 
model, such as a path continuous process that interpolates $Z$ between jump
times is used. This interpolation may be a straight line joining the 
nodes involved in the jump or a piecewise straight line that
bounces off obstructions. We state a very general result that is useful in
this context. 
In particular, we note that an interpolated process arising from above
modeling situation will satisfy the 
condition \eqref{eq:IandZ} required by Theorem \ref{thm:interpI} because the
maximum deviation of $I$ from $Z$ will be bounded in terms of the size of the periodic
cell. 

\begin{theorem}\label{thm:interpI}
Suppose $I$ is a {\em cadlag} process defined on $(\Omega,\mathcal{F},\mathbb{P})$ 
and that there exists $k >0$ such that 
\begin{equation}\label{eq:IandZ}
|I(t)-Z(t)| \leq k \quad \forall t \geq 0,
\end{equation}     
almost surely. 
Then for fixed $0<t_0<T<\infty$ 
\[
\lim_{\eps \to 0} \sup_{t \in [t_0,T]} |\eps^2 I(t/\eps^2) - \bar{U}t| = 0,
\]
almost surely. Moreover if we define $I_\eps$ by 
$I_\eps(t) = \eps (I(t/\eps^2) - \bar{U}t/\eps^2)$ then
\[
I_\eps \Rightarrow {\bf Z} \text{ in } D^{d}[0,\infty) \text{ as } \eps \rightarrow 0,
\]
where ${\bf Z}$ is as defined in Theorem \ref{thm:deff_rigorous}. 
\end{theorem}
\begin{proof}
The almost sure limit follows readily since for all $t \geq 0$
\[
|\eps^2 I(t/\eps^2) - \bar{U}t| \leq |\eps^2 Z(t/\eps^2) - \bar{U}t| + \eps^2
k.
\]
To see the second limit, we note that for all $t \geq 0$
\[
|I_\eps(t) - Z_\eps(t)| \leq \eps k,
\]
and hence Theorem 3.1 of \cite{billingsley2011convergence} in conjunction with
Theorem \ref{thm:deff_rigorous} delivers the result.
\end{proof}

%% file: inputs/solvability-conditions.tex
\section{Null drift conditions}\label{sect:solvability}

We take a closer look at the null drift condition \eqref{eq:null-drift} in this section.

\subsection{Detailed balance and null drift}

If $e'=(y,x) \in \sE$ and $e=(x,y) \in \sE$, we shall say
$e'$ is the {\em reversal} of $e$. In general, an edge $e \in \sE$ may not
have a reversal (in $\sE$).  We shall say that the process $Z$
(or the triple $(\sS,\sE,\lambda)$)
satisfies {\em detailed balance} provided
\begin{equation}\label{eqn:detailed-balance}
	\lambda_e \pi(\partial_- e) = \lambda_{e'} \pi(\partial_- e') \quad
        \forall e \in \sE,
\end{equation}
where $e'$ denotes the reversal of $e$ and  $\pi$ is the (pull back of the) stationary distribution
of the projected process $Y$. We note that the
requirement that the reversal of each edge (in $\sE$) is also present (in $\sE$)
is implicit in our definition of
detailed balance.
We also note that our notion of detailed balance is stricter than the
conventional notion of detailed balance applied to $Y$, which merely requires
that $L(x,y) \pi(x) = L(y,x) \pi(y)$ for all $x,y \in \oS$.

Recall the map $\kappa:\oE \to \sE_\Pi$ defined in
Section~\ref{sect:preliminaries} which is a surjection.
We shall say that $\kappa$ {\em commutes with edge reversals} provided
whenever $e = (y_1,y_2) \in \sE_\Pi$ and its reversal $e'=(y_2,y_1) \in \sE_\Pi$
then for each $\tilde{e} \in \kappa^{-1}(\{e\}) \subset \oE$
we have $\tilde{e}' \in \kappa^{-1}(\{e'\}) \subset \oE$, where
$\tilde{e}'$ is the reversal of $\tilde{e}$.

We state a useful lemma.
\begin{lemma}\label{lem-detailed-equivalence}
Suppose $\kappa$ is a bijection and it commutes with edge reversals.
Then detailed balance of $Z$, defined by \eqref{eqn:detailed-balance}, is equivalent to the
standard notion of detailed balance of $Y$.
\end{lemma}
\begin{proof}
Given $e=(z,z') \in \sE$, set $y=\Pi(z)$ and $y'=\Pi(z')$. Then
$\kappa([e])=(y,y')$ where $[e] \in \oE$ is the equivalence class of $e$.
Our assumptions on $\kappa$ imply that $L(y,y')=\lambda_e$ and
$L(y',y)=\lambda_{e'}$, where $e'=(z',z) \in \sE$ is the reversal of $e$.
\end{proof}
We note that, for graphs that are
of interest to us, $\kappa$ will satisfy the requirements of this lemma.

\begin{theorem}\label{thm:detailed-balance}
Suppose that $Z$ satisfies the detailed balance condition~(\ref{eqn:detailed-balance}).
Then the null drift condition holds. As a special case, if the rates are
reversible, then the null drift condition holds. 
\end{theorem}
\begin{proof}
From the detailed balance we obtain that
\[
\sum_{e \sE_y} \nu_e \lambda_e \pi(\partial_- e)= \sum_{e \in \sE_y} \nu_e \lambda_{e'}
\pi(\partial_- e'),
\]
where $e'$ stands for the reversal of an edge $e$.  Thus
\[
\begin{aligned}
\bar{U} &= \sum_{y \in \oS} \sum_{e \in \sE_y} \nu_e \lambda_e \pi(\partial_-
e) = \sum_{y \in \oS} \sum_{e \in \sE_y} \nu_e \lambda_{e'} \pi(\partial_- e')
= - \sum_{y \in \oS} \sum_{e \in \sE_y} \nu_{e'} \lambda_{e'} \pi(\partial_-
e')\\
&= -\sum_{y \in \oS} \sum_{e \in \sE'_y} \nu_e \lambda_e \pi(\partial_- e)
= -\sum_{y \in \oS} \sum_{e \in \sE_y} \nu_e \lambda_e \pi(\partial_- e) = -\bar{U},
\end{aligned}
\]
where we have used the fact that $\nu_{e'}=-\nu_e$ and \eqref{eqn:sum-of-sums}.

In the special case of reversible rates, the detailed balance holds since the
stationary distribution is constant. Thus the null drift condition follows. 
\end{proof}

Note that the reverse implication of Theorem~\ref{thm:detailed-balance} does not hold. As a counter example, consider the ``whirlpool'' graph depicted in Figure~\ref{fig:detailed-balance}. Here,
\begin{equation*}
\begin{aligned}
	\bar{S} &= \{(0,0),(1/3,0),(2/3,0),(0,1/3),(2/3,1/3),(0,2/3),(1/3,2/3),(2/3,2/3)\} \\
	&\eqqcolon \{y_1,y_2,y_3,y_4,y_5,y_6,y_7,y_8\}.
\end{aligned}
\end{equation*}
The edge set can easily be deduced from Figure~\ref{fig:detailed-balance}. The
edge set $\oE$ is depicted. Every edge has the same jump rate
$\bar{\lambda} >0$. Note that the stationary distribution
is given by $\pi(y) = 1/8$ for all $y \in \oS$. It is also clear that the detailed balance
condition does not hold.

Simple algebra shows that the drift field $\rho=(\rho_1,\rho_2)^T$ is given by
\begin{equation*}
\begin{aligned}
\rho_1(y_i) \coloneqq
\begin{cases}
	 \bar{\lambda} & i = 1,2,3\\
	   0 & i = 4,5\\
	-\bar{\lambda} & i = 6,7,8
\end{cases}
\end{aligned}
\end{equation*}
and
\begin{equation*}
\begin{aligned}
\rho_2(y_i) \coloneqq
\begin{cases}
	 \bar{\lambda} & i = 3,5,8\\
	   0 & i = 2,7\\
	-\bar{\lambda} & i = 1,4,6.
\end{cases}
\end{aligned}
\end{equation*}
Clearly, $\sum_{y \in \oS} \rho(y) \pi(y)= 0$ and thus the null drift
condition holds.
\begin{figure}
   \centering
   \includegraphics[width=60mm]{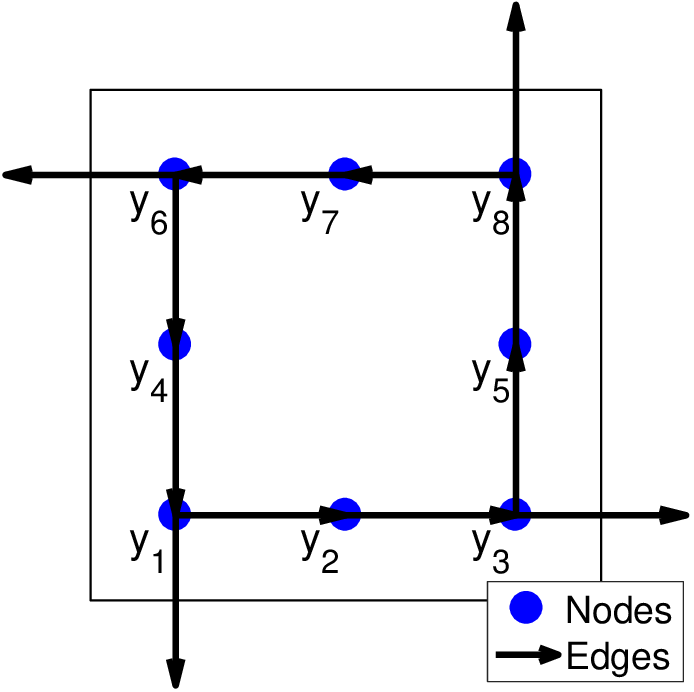}
   \caption{The ``whirlpool'' example has null drift and a rate matrix $L$ that is not symmetric with respect to $\pi$. All edges have the same jump rate $\bar{\lambda}$.}
   \label{fig:detailed-balance}
\end{figure}

Here we mention a connection between detailed balance and self-adjointness
of $L$.
Recall that $L$ is
self-adjoint with respect to a probability measure $p$ on $\oS$ provided
\[
\sum_{y\in\oS} (Lf)(y) g(y) p(y) = \sum_{y\in\oS} f(y) (Lg)(y) p(y)
\]
for every pair $f,g:\oS \to \real$.
The following lemma is well known.
\begin{lemma}\label{lem:detailed-balance-symmetry}
The process $Y$ satisfies the detailed balance condition if and only if $L$ is
symmetric with respect to the stationary distribution $\pi$.
\end{lemma}

\subsection{Sufficient symmetries of the graph imply null drift}
Since the verification of the null drift condition involves the knowledge of
$\pi$, it may be useful to find conditions that guarantee null drift without
having to compute $\pi$. We shall provide some results in this subsection that
show that if the graph has sufficient symmetries then the null drift condition
holds.

Given any vector space $V$ over $\real$, let $\sL(V)$ denote the vector space (over $\real$) of all continuous linear maps from $V$ into $V$. Let $\phi:\oS \to \oS$ be a bijection. We associate with it the {\em pull-back} $\phi^* \in \sL(\real^{\oS})$ defined by
\begin{equation*}
\phi^* f = f \circ \phi \quad \forall f \in \real^{\oS}.
\end{equation*}
We state some basic lemmas about $\phi^*$ and ultimately arrive at a
sufficient condition for null drift. Given $f,g \in \real^{\oS}$ we
define their inner product in the obvious way:
\[
(f,g) = \sum_{y \in \oS} f(y) g(y).
\]
\begin{lemma}\label{lem:phistar-transpose}
Let $\phi:\oS \to \oS$ be a bijection. Then $(\phi^*)^T = (\phi^{-1})^*$ and $\phi^*$ preserves the inner product: $(\phi^* f,\phi^* g) = (f,g)$ for all $f,g \in \real^{\oS}$.
\end{lemma}

\begin{lemma}
Let $\phi:\oS \to \oS$ be a bijection and $A \in
\sL(\real^{\oS})$.
The following are equivalent:
\begin{enumerate}
\item $A(\phi(x),\phi(y)) = A(x,y) \quad \forall x,y \in \oS$.
\item $A \circ \phi^* = \phi^* \circ A$.
\item $A^T \circ (\phi^{-1})^* = (\phi^{-1})^* \circ A^T$.
\end{enumerate}
\end{lemma}
\begin{proof}
It is adequate to note that $\phi^*: \oS \times \oS \rightarrow \real$ is given by
\begin{equation*}
\phi^*(x,y) = (\phi^* \delta_y)(x) =
\begin{aligned}
\begin{cases}
	1 & \phi(x) = y\\
  	0 & \text{otherwise.}
\end{cases}
\end{aligned}
\end{equation*}
\end{proof}

If a bijection $\phi$ satisfies the conditions in the above lemma, then $\phi$ or the
associated $\phi^*$ is called a {\em symmetry} of $A \in \sL(\real^{\oS})$.
The following lemma is immediate.
\begin{lemma}\label{lem-phi-phiinv}
If the bijection $\phi:\oS \to \oS$ is a symmetry of $A \in \sL(\real^{\oS})$
then $\phi^{-1}$ is also a symmetry of $A$.
\end{lemma}

The previous lemmas apply to very general $A$. Now we focus on the case when $A$ (and hence $A^T$) has a one-dimensional null space, which is the case when $A$ is the generator of an ergodic Markov process.

\begin{lemma}
Let $A \in \sL(\real^{\oS})$ and suppose $A$ has a one-dimensional null space. Suppose $\phi$ is a symmetry of $A$.  Let $f \in \real^{\oS}$ span the null space of $A^T$. Then $\phi^* f = f$.
\end{lemma}
\begin{proof}
Because $\phi^{-1}$ is also a symmetry of $A$ by Lemma \ref{lem-phi-phiinv},
we have
\begin{equation*}
\begin{aligned}
A^T \circ \phi^* f = \phi^* \circ A^T f = 0.
\end{aligned}
\end{equation*}
That is, $\phi^* f$ is in the null space of $A^T$. Because $A^T$ has a one-dimensional null space, we have $\phi^* f = \alpha f$ for some $\alpha \neq 0$. Since $\phi$ is a bijection on $\oS$, $\alpha = 1$.
\end{proof}

Before we present the main result, we observe that we can write the null drift
condition as
\begin{equation*}
(\rho_i,\pi) = 0, \quad i=1,\dots,d,
\end{equation*}
where $\rho$ is the drift field.

\begin{theorem}\label{thm:sym-solvability}
Consider a graph $(\sS,\sE,\lambda)$ with generator (rate matrix) $L \in
\sL(\real^{\oS})$ (\ref{eqn:rate-matrix}) and stationary distribution
$\pi$~(\ref{eqn:stat-dist}). Suppose that for each $i=1,\dots,d$ there exists
a symmetry $\phi_i$ of $L$ such that $\phi_i^* \rho_i = -\rho_i$ and
$\phi_i^{-1} = \phi_i$. Then the null drift condition holds.
\end{theorem}
\begin{proof}
For each $i=1,\dots,d$
\begin{equation*}
(\rho_i,\pi) = (\phi_i^* \rho_i,\phi_i^* \pi) = (-\rho_i, \pi).
\end{equation*}
Thus $(\rho_i,\pi)=0$.
\end{proof}
\begin{remark}
	Note that if we have a reflection symmetry about each direction $i$ and denote the reflection map by $\phi_i$, then the assumptions of Theorem~\ref{thm:sym-solvability} are satisfied.
\end{remark}

%% file: inputs/formal-derivation.tex
\section{Formal asymptotics}\label{sect:formal-derivation}
The homogenization result for PDEs is often motivated by a formal asymptotic
expansion, which leads to the so-called unit-cell
problems~\cite{keller1977effective,bensoussan2011asymptotic}. While this
procedure does not constitute a rigorous proof, it can be a useful tool.
In this section, we shall derive an analogous asymptotic expansion leading to unit-cell problems and a formula for the effective diffusivity.
\subsection{Anzats}
Denote the probability mass function of $Z_\eps(t)$ by $p_\eps(t,z)$, where $z \in \eps \sS$. The Kolmogorov forward equations are:
\begin{equation}\label{eqn:p_evolve}
\frac{\partial p_\eps}{\partial t}(t,z)
= \frac{1}{\eps^2} \sum_{e \in \E_y'} p_\eps(t,z-\eps \nu_e) \lambda_e
- \frac{1}{\eps^2} \lambda^0(y) p_\eps(t,z)
\end{equation}
for all $z \in \eps \sS$. We shall seek an asymptotic expansion of the form
\begin{equation}\label{eqn:p_asymp}
p_\eps(t,z) = p_0(t,z,y) + \eps p_1(t,z,y) + \eps^2 p_2(t,z,y) + \dots,
\end{equation}
where $p_j:[0,\infty) \times \real^d \times \oS \rightarrow [0,\infty)$ and $y = \Pi (z /\eps) \in \oS$. We assume that $p_j$ is sufficiently differentiable in $z$ for $j=0,1,2$. The machinery behind this asymptotic analysis involves substitution of the expansion (\ref{eqn:p_asymp}) into the forward equations (\ref{eqn:p_evolve}) and Taylor expansion about $z$, which yields:
\begin{equation}\label{eqn:asymp_taylor}
\begin{aligned}
\frac{\partial p_0}{\partial t}(t,z)
=&\; \frac{1}{\eps^2} \sum_{e \in \E_y'} \Big(p_0(t,z,y-\nu_e)\lambda_e + \eps p_1(t,z,y-\nu_e)\lambda_e + \eps^2 p_2(t,z,y-\nu_e)\Big)\lambda_e \\
&- \frac{1}{\eps} \sum_{e \in \E_y'} \Big(\nu_e^T \nabla_z  p_0(t,z,y-\nu_e)\lambda_e + \eps \nu_e^T \nabla_z  p_1(t,z,y-\nu_e)\Big)\lambda_e \\
&+ \frac{1}{2}\sum_{e \in \E_y'} \nu_e^T D_{zz} p_0(t,z,y-\nu_e) \nu_e \lambda_e \\
&- \frac{1}{\eps^2} \lambda^0(y) \Big(p_0(t,z,y) + \eps p_1(t,z,y) + \eps^2 p_2(t,z,y)\Big) + O(\eps).
\end{aligned}
\end{equation}
We note that we use $\nabla_z$ for the first derivative and $D_{zz}$ for the second derivative so that $D_{zz} p_0$ is the Hessian. Equating like powers of $\eps^{-2}, \eps^{-1}$, and $\eps^{0}$ in~(\ref{eqn:asymp_taylor}) yields a hierarchy of problems. Solving these problems yields the homogenized equation governing $p_0$, which is (formally) the probability density function describing the limiting motion of $Z_{\eps}$ as $\eps \rightarrow 0$.
\subsection{Stationary distribution}
Equating like powers of $\eps^{-2}$ in~(\ref{eqn:asymp_taylor}), we have:
\begin{equation*}
0 = - \lambda^0(y) p_0(t,z,y) + \sum_{e \in \sE_y'} p_0(t,z,y-\nu_e) \lambda_e,
\end{equation*}
for all $z \in \eps \sS$ and $y = \Pi (z /\eps) \in \oS$. This can be written as $L^T p_0(t,z,.) = 0$. Since $(\oS,\sE_{\Pi},\lambda)$ is a strongly connected graph, $Y$ is an ergodic Markov process and $L^T$ has a one-dimensional null space spanned by the stationary distribution $\pi \in \real^{\oS}$. Thus for any fixed $t,z$, there exists a constant $\bar{p}(t,z)$ such that
\begin{equation}\label{eqn:p0}
p_0(t,z,y) = \bar{p}(t,z)\pi(y),
\end{equation}
where
\begin{equation}\label{eqn:stat-dist}
(L^T \pi)(y) = 0, \quad
\sum_{y \in \oS} \pi(y) = 1.
\end{equation}


\subsection{Unit-cell problem}
Equating like powers of $\eps^{-1}$ in~(\ref{eqn:asymp_taylor}), we have
\begin{equation*}
\begin{aligned}
0 &= \sum_{e \in \E_y'} p_1(t,z,y-\nu_e) \lambda_e - \sum_{e \in \E_y'} \nu_e^T \nabla_z  \bar{p}(t,z)\lambda_e \pi(y-\nu_e) - \lambda^0(y) p_1(t,z,y)
\end{aligned}
\end{equation*}
for all $z \in \eps \sS$ and $y = \Pi (z /\eps) \in \oS$. Defining $\sigma \in (\real^d)^{\oS}$ by
\begin{equation}\label{eqn:sigma}
\sigma(y) \coloneqq \sum_{e \in \E_y'}  \nu_e \lambda_e \pi(y-\nu_e),
\end{equation}
this can be written as
\begin{equation*}
(L^T p_1(t,z,\cdot))(y) = \nabla_z  \bar{p}(t,z)^T \sigma(y).
\end{equation*}
Hence the relationship between $p_1$ and $\bar{p}$:
\begin{equation}\label{eqn:pbar-p1-relationship}
p_1(t,z,y) = \nabla_z  \bar{p}(t,z)^T \omega(y),
\end{equation}
where $\omega \in (\real^d)^{\oS}$ is a solution to the \emph{unit-cell problem}:
\begin{equation}\label{eqn:unitcell}
(L^T \omega)(y) = \sigma(y) \quad \text{for all } y \in \oS.
\end{equation}
The following lemma states that the unit-cell problem is solvable
if and only if null drift condition $\bar{U}=0$ holds.

\begin{lemma}\label{lem:unitcell-solvability}
The unit-cell problem~(\ref{eqn:unitcell}) is solvable if and only if null
drift condition holds. That is, if and only if
\begin{equation}\label{eqn:unitcell_solv}
\sum_{y \in \oS} \sum_{e \in \E_y}  \nu_e \lambda_e \pi(y) = 0.
\end{equation}
\end{lemma}
\begin{proof}
In order for the unit-cell problem to be solvable, it is necessary and sufficient that $\sigma$ be orthogonal to the null space of $L$. Because the null space of $L$ consists of constant functions, this is equivalent to $\sum_{y \in \oS} \sigma(y) = 0$. That is,
\begin{equation*}
\sum_{y \in \oS} \sum_{e \in \E_y'}  \nu_e \lambda_e \pi(y-\nu_e) =0.
\end{equation*}
Using (\ref{eqn:sum-of-sums}) and noting $y - \nu_e = \partial_-e$ for all $e \in \sE_y'$, this is equivalent to
\begin{equation*}
\sum_{y \in \oS} \sum_{e \in \E_y}  \nu_e \lambda_e \pi(y) =0.
\end{equation*}
\end{proof}

\subsection{Effective diffusivity}
We now show $\bar{p}$ is governed by the diffusion equation with effective diffusivity matrix, $K \in \real^{d \times d}$:
\begin{equation}
  \frac{\partial \bar{p}}{\partial t}(t,z) = \text{div}_z \cdot (K \nabla_z \bar{p}(t,z)).
\end{equation}
This suggests that $Z_{\eps}$ tends to a Brownian motion whose covariance matrix is given by $2Kt$. Before proceeding, we introduce an easily verifiable lemma.
\begin{lemma}\label{lem:trace}
For any symmetric $A \in \real^{d \times d}$ and $x,y \in \real^d$,
\begin{equation*}
  \mathrm{trace}\big(A D_{zz}\bar{p}(t,z)\big) = \textup{div}_z  \big(A \;\nabla_z  \bar{p}(t,z)\big)
\end{equation*}
and
\begin{equation*}
  \mathrm{trace}(A x y^T) = x^T A y.
\end{equation*}
\end{lemma}

Balancing the $\eps^0$ terms in~(\ref{eqn:asymp_taylor}) yields
\begin{equation*}
\begin{aligned}
\frac{\partial \bar{p}}{\partial t}(t,z) \pi(y)
&= (L^T p_2(t,z,.))(y) - \sum_{e \in \E_y'} \nu_e^T \nabla_z  p_1(t,z,y-\nu_e) \lambda_e \\
&+ \frac{1}{2} \sum_{e \in \E_y'} \nu_e^T D_{zz} \bar{p}(t,z) \nu_e \lambda_e \pi(y-\nu_e)
\end{aligned}
\end{equation*}
for all $z \in \eps \sS$ and $y = \Pi (z /\eps) \in \oS$. Since the range of $L^T$ consists of functions that sum to zero, summing the above over $y \in \oS$ eliminates $p_2$ terms and yields
\begin{equation*}
\frac{\partial \bar{p}}{\partial t}(t,z)
= \frac{1}{2} \sum_{y \in \oS} \sum_{e \in \E_y'} \nu_e^T D_{zz} \bar{p}(t,z) \nu_e \lambda_e \pi(y-\nu_e)
- \sum_{y \in \oS} \sum_{e \in \E_y'} \nu_e^T D_{zz} \bar{p}(t,z) \omega(y-\nu_e) \lambda_e,
\end{equation*}
where we have used (\ref{eqn:pbar-p1-relationship}). Use (\ref{eqn:sum-of-sums}) to rewrite this as
\begin{equation*}
\begin{aligned}
\frac{\partial \bar{p}}{\partial t}(t,z)
= \frac{1}{2} \sum_{y \in \oS} \sum_{e \in \E_y} \nu_e^T D_{zz} \bar{p}(t,z) \nu_e \lambda_e \pi(y)
- \sum_{y \in \oS} \sum_{e \in \E_y} \nu_e^T D_{zz} \bar{p}(t,z) \omega(y) \lambda_e.
\end{aligned}
\end{equation*}
Then by symmetry of $D_{zz} \bar{p}(t,z)$ and Lemma~\ref{lem:trace},
\begin{equation*}
\begin{aligned}
  \frac{\partial \bar{p}}{\partial t}(t,z)
  &= \frac{1}{2} \mathrm{trace}\bigg[D_{zz}\bar{p}(t,z) \sum_{y \in \oS} \sum_{e \in \E_y} \Big(
    \nu_e \nu_e^T \lambda_e \pi(y) - \nu_e \omega(y)^T \lambda_e - \omega(y) \nu_e^T \lambda_e\Big)\bigg] \\
  &= \text{div}_z  (K \nabla_z  \bar{p}(t,z)),
\end{aligned}
\end{equation*}
where
\begin{equation}\label{eqn:K}
K \coloneqq \frac{1}{2} \sum_{y\in\oS} \sum_{e \in \sE_y} \Big(\nu_e \nu_e^T  \lambda_e \pi(y) - \nu_e \omega(y)^T\lambda_e - \omega(y)\nu_e^T\lambda_e\Big).
\end{equation}
In Appendix A, we show that $K=C$ where $C$ is the diffusivity obtained by the
rigorous result.

%% file: inputs/variational-formulation.tex
\section{Variational formulation for the reversible rates  case $L = L^T$}\label{sect:variational-formulation}
When the rates are reversible, the generator $L$ is symmetric. We note that
this also means that reverse edges are present: if $(x,y) \in \oE$ then
$(y,x) \in \oE$. In the
continuous (diffusion PDE) homogenization case, where the analogous differential operator $L$ is symmetric, the effective diffusivity may be characterized by a variational formulation \cite{kozlov1980homogenization}. Likewise in the discrete lattice ($\integ^d$) setting with reversible rates, a variational characterization may be found in \cite{caputo2003finite}.

We shall derive a variational characterization of the effective diffusivity in
our graph setting for the symmetric (reversible rates) case. Thus we assume $L
= L^T$ throughout this section. In the continuous case of the diffusion PDE, the variational characterization involves the use of results from
vector calculus, such as the divergence theorem and the fact that the
generator
can be written as $Lf = \textup{div}(A \nabla f)$ where $A$ is the (local)
diffusivity matrix. In the lattice
case~\cite{caputo2003finite}, due to the Cartesian structure, the discrete
analog of the gradient and divergence operators are relatively easy to
define. 

In the general graph setting considered here, we develop the analogous
operators of gradient and divergence and an equivalent of the local diffusivity
$A$ in order
to provide the variational formulation. In our setting, the gradient operator
$\nabla$ applies to node functions, the divergence applies to
$\real^d$-valued edge functions (edge vector fields) and $A$ is a matrix valued function of the
edges. Our approach will be self-contained and limited to our goal of 
deriving a variational characterization in terms of divergence and gradient 
analogous to the diffusion PDE case. There is substantial literature 
on the development of what is termed ``discrete calculus'', see \cite{grady2010discrete}
 for a textbook introduction. There is also substantial overlapping literature 
under the title {\em discrete exterior calculus} which develops a discrete 
analog of the exterior calculus of differential forms on manifolds. 
While the domain of exterior differential calculus is the differentiable
manifold, the discrete exterior calculus is developed on simplicial complexes. See for instance
\cite{hirani2003discrete, desbrun2005discrete}. Some motivation for the
development of discrete calculus on graphs has also stemmed from applications
in machine learning, see  
\cite{zhou2005learning} for instance. Most of the works in the literature on discrete calculus 
do not regard the underlying graphs or simplicial complexes as embedded 
in a Euclidean space or a manifold and therefore are purely topological. 
Our approach differs in that aspect since we regard the graph as embedded 
in $\real^d$ and thus our calculus includes the geometry via edge displacements $\nu_e$ as well as topology.
The work in \cite{friedman2004calculus} regards the edges of the graph as 
one dimensional line segments and hence the calculus developed does involve 
geometry. The approach in \cite{friedman2004calculus} appears 
to differ 
from ours in that while our edges have an associated displacement 
$\nu_e \in \real^d$, we do not speak of points on the interior of an edge. 
       

Throughout this section, we use the same notation for the inner products
defined on $\real^{\oS}$ and $(\real^d)^{\oE}$ by
\begin{equation}
(f,g) = \sum_{y \in \oS} f(y) g(y)
\end{equation}
for $f,g \in \real^{\oS}$ and
\begin{equation}
(f,g) = \sum_{e \in \oE} f(e)^T g(e)
\end{equation}
for $f,g \in (\real^d)^{\oE}$. The space in which an inner product applies 
should be clear from the context. 

\subsection{Gradient and divergence}
We introduce gradient and divergence as they apply
to our graph setting. We note that our definitions of gradient and divergence depend 
only on the topology and geometry (via the embedding in $\real^d$) of the graph. Thus the graph connectivity as
well as $\nu_e$ play a role in the definition, but the jump rates (weights) 
$\lambda_e$ do not enter into these definitions. Thus our definitions differ 
from that of \cite{zhou2005learning} where jump rates are incorporated into the
definitions and moreover \cite{zhou2005learning} does not regard the graph as
embedded in a Euclidean space.  

Let the \emph{gradient} $\nabla: \real^{\oS} \rightarrow (\real^d)^\E$ be defined by
\begin{align}
(\nabla f)(e) = (f(\partial_+ e) - f(\partial_- e))\frac{\nu_e}{|\nu_e|^2}
  \quad \forall f \in \real^{\oS}.
\end{align}
We observe that $(f(\partial_+ e) - f(\partial_- e))/|\nu_e|$ is the 
divided difference of $f$ along the edge $e$ and this is multiplied by the unit vector
$\nu_e/|\nu_e|$ to obtain $(\nabla f)(e)$. 

Let the \emph{divergence} $\textup{div}: (\real^d)^{\oE} \rightarrow \real^{\oS}$ be defined by
\begin{equation}
(\textup{div} f)(y) = \sum_{e \in \E_y} \frac{\nu_e^T f(e)}{|\nu_e|^2} -
  \sum_{e \in \E'_y} \frac{\nu_e^T f(e)}{|\nu_e|^2} \quad \forall f \in (\real^d)^{\oE}.
\end{equation}
A rough explanation for the intuition behind our definition of divergence is as follows. 
Imagine a ``vanishingly thin'' finite cylinder associated to each edge $e$ so that the cylinder 
has as its axis the line joining the end points of the edge with the length 
of the axis given by $|\nu_e|$. 
The ``flux'' due to $f$ along an edge $e$ can be thought of as flowing 
along the cylinder and proportional to $\nu_e^T f(e)/|\nu_e|$, the component of
$f(e)$ along the axis. In order to compute the contribution of this flux 
to the divergence at $y=\partial_-e$, we normalize this flux by the ratio of the volume of
the cylinder to the area of its cross section. This leads to the additional factor $1/|\nu_e|$. 
We note that the incoming fluxes at a node $y$ are to be subtracted from the
sum of outgoing fluxes from $y$, resulting in the above formula.

\begin{lemma}[Divergence theorem]\label{lem:divergence}
Let $f \in (\real^d)^{\oE}$ and $g \in \real^{\oS}$. 
Then
\begin{equation*}
(f,\nabla g) = -(\textup{div}f,g).
\end{equation*}
\end{lemma}
\begin{proof}
We have
\[
\begin{aligned}
\sum_{y\in\oS} &(\textup{div}f)(y) g(y) = \sum_{y \in \oS} \sum_{e \in \sE_y}
\frac{\nu_e^T f(e)}{|\nu_e|^2} g(y) - \sum_{y \in \oS} \sum_{e \in \sE'_y}
\frac{\nu_e^T f(e)}{|\nu_e|^2} g(y)\\
&=   \sum_{e \in \oE}
\frac{f(e)^T \nu_e}{|\nu_e|^2} g(\partial_-e) - \sum_{e \in \oE} 
\frac{f(e)^T \nu_e}{|\nu_e|^2} g(\partial_+e)\\
&= - \sum_{e \in \oE} \frac{f(e)^T \nu_e}{|\nu_e|^2}
(g(\partial_+e)-g(\partial_-e))
 = -\sum_{e \in \oE} f(e)^T (\nabla g)(e).
\end{aligned}
\]
\end{proof}

\begin{remark} 
The definitions of gradient and divergence as well as the divergence theorem 
do not depend on the jump rates (edge weights) $\lambda_e$. 
\end{remark} 

\subsection{Unit-cell problem in terms of divergence and gradient} 
In the diffusion PDE setting, where the generator $L$ is symmetric ($L^T=L$), 
one may write 
\[
Lf =  \textup{div}(A (\nabla f)),
\]
where $A$ is the diffusivity matrix. We derive the same expression for 
the generator of the process $Z$ in the case of reversible rates.  

We define the matrix valued edge function $A: \oE \rightarrow \real^{d\times d}$ by
\begin{equation}
	A(e) = \frac{1}{2} \nu_e \nu_e^T \lambda_e,
\end{equation}
which will be used throughout this section. 
We observe that one may regard the {\em local diffusivity matrix} at $y \in
\oS$ of the process $Z$ to 
be given by half the conditional rate of change of covariance,
\[
\frac{1}{2} \lim_{h \to 0+} \textup{Cov}((Z(t+h)-Z(t)) /h \, | \, Z(t)=y).
\] 
Noting that 
\[
\begin{aligned}
\textup{Cov}(Z(t+h)-Z(t)) &= \Exp\left((Z(t+h)-Z(t))(Z(t+h)-Z(t))^T\right)\\
&\hspace{2em} -\Exp(Z(t+h)-Z(t)) \, \Exp(Z(t+h)-Z(t))^T\\ &= \sum_{e \in \sE_y} \nu_e \nu_e^T
\lambda_e h + o(h) \quad \text{ as } h \to 0+,
\end{aligned}
\]
we obtain that 
\[
\frac{1}{2} \lim_{h \to 0+} \textup{Cov}((Z(t+h)-Z(t)) /h \, | \, Z(t)=y) =
\sum_{e \in \sE_y} A(e).
\]
Thus $A(e)$ is the contribution to the local diffusivity at the node $\partial_-e$ 
from the edge $e$.  

\begin{lemma}\label{lem:Asymm}
If $f,g \in (\real^d)^{\oE}$, then
\[
(f,Ag) = (Af,g),
\]
where $Af$ and $Ag$ refer to the $\real^d$ valued edge functions 
given by $(Af)(e)=A(e)f(e)$ and $(Ag)(e)=A(e)g(e)$. 
\end{lemma}
\begin{proof} 
\[
(f,Ag) = \frac{1}{2} \sum_{e \in \oE} f(e)^T \nu_e \nu_e^T \lambda_e g(e) =
\frac{1}{2} \sum_{e \in
  \oE} g(e)^T \nu_e \nu_e^T \lambda_e f(e) = (g,Af)=(Af,g).
\]
\end{proof}
 
\begin{lemma}\label{lem:Ldivgrad}
Under the reversible rates assumption ($L^T=L$) we have
\begin{equation}
(Lf)(y) = \textup{div}(A \nabla f)(y),
\end{equation}
for all $f \in \real^{\oS}$ and $y \in \oS$. 
\end{lemma}  
\begin{proof}
In this proof $e'$ denotes the reversal of an edge $e$. We have
$\nu_{e'}=-\nu_e$ and $(\nabla f)(e')=(\nabla f)(e)$, and under the reversible
rates assumption we also have that $A(e')=A(e)$.  Thus
\[
\begin{aligned}
\textup{div}(A \nabla f)(y) &=  \sum_{e \in \sE_y} \frac{\nu_e^T A(e)
  (\nabla f)(e)}{|\nu_e|^2} -   \sum_{e \in \sE'_y} \frac{\nu_e^T A(e)(\nabla
  f)(e)}{|\nu_e|^2}\\
&=  \sum_{e \in \sE_y} \left(\frac{\nu_e^T A(e)(\nabla
  f)(e)}{|\nu_e|^2}-\frac{\nu_{e'}^T A(e')(\nabla f)(e')}{|\nu_{e'}|^2}\right)\\
&= 2 \sum_{e \in \sE_y} \frac{\nu_e^T A(e)(\nabla f)(e)}{|\nu_e|^2}\\
&=  \sum_{e \in \sE_y} (f(\partial_+e)-f(\partial_-e)) \lambda_e =  (Lf)(y).
\end{aligned}
\]
\end{proof}

Given $\xi \in \real^d$ we define the {\em modified unit-cell problem} to be 
the problem of finding $\Upsilon_{\xi} \in \real^{\oS}$ such that
\begin{equation}\label{eqn:unitcell2}
(L^T\Upsilon_{\xi})(y) = \xi^T\sigma(y) \text{ for all } y \in \oS.
\end{equation}
Solving the modified unit-cell problem for a basis of vectors $\xi$ is 
equivalent to solving the original problem \eqref{eqn:unitcell} since
by the linearity of the unit-cell problem (\ref{eqn:unitcell}),
$\Upsilon_{\xi}(y) = \xi^T \omega(y)$ where $\omega\in \real^{\oS}$ is a
solution to the unit-cell problem.

The following theorem formulates the modified unit-cell problem in terms of $\nabla$ and div.
\begin{theorem}[Unit-cell]\label{t:unitcell}
Let $\xi \in \real^d$. Then $\Upsilon_{\xi}\in \real^{\oS}$ solves the
modified unit-cell problem \eqref{eqn:unitcell2}
if and only if $\Upsilon_{\xi}$ satisfies
\begin{equation}\label{eqn:unitcelldiv}
\textup{div}\left(A \left(\nabla\Upsilon_{\xi} + \frac{1}{|\oS|}\xi\right)\right)(y) = 0 \text{ for all } y \in \oS.
\end{equation}

\end{theorem}
\begin{proof}
It is adequate to show that $-|\oS| \xi^T \sigma(y) = \textup{div}(A \xi)(y)$. 
Noting that $\pi = 1/|\oS|$, we have
\[
\begin{aligned}
- |\oS| \xi^T \sigma(y)
&= -\sum_{e \in \sE_y'} |\oS| \lambda_e \pi(y-\nu_e) \nu_e^T \xi\\
&= -\sum_{e \in \sE_y'} \lambda_{e} \nu_{e}^T \xi
= \sum_{e \in \sE_y} \lambda_e \nu_e^T \xi,
\end{aligned}
\]
where we have used the reversible rates assumption in the last step. 
On the other hand, noting that $e'$ denotes the reversal of an edge $e$,
we have
\[
\begin{aligned}
\textup{div}(A \xi)(y) &= \sum_{e \in \sE_y} \frac{\xi^T A^T(e)
  \nu_e}{|\nu_e|^2} - \sum_{e \in \sE'_y} \frac{\xi^T A^T(e)
  \nu_e}{|\nu_e|^2}\\
&= \sum_{e \in \sE_y} \left( \frac{\xi^T A^T(e) \nu_e}{|\nu_e|^2} -
\frac{\xi^T A^T(e') \nu_{e'}}{|\nu_{e'}|^2}\right)\\
&= 2 \sum_{e \in \sE_y} \frac{\xi^T A^T(e) \nu_e}{|\nu_e|^2}
= \sum_{e \in \sE_y} \xi^T \nu_e \lambda_e.
\end{aligned}
\]
\end{proof}

\subsection{Minimization of energy}
Having reformulated the unit-cell problem 
in terms of gradient, divergence and local diffusivity, we are ready to
state the main theorem of this section. 

For any fixed $\xi \in \real^d$ we define the associated {\em energy} function
$E_{\xi}:\real^{\oS} \to \real$ by
\begin{equation}
E_{\xi}(f) = \frac{1}{2}\left(A \left(\nabla f +
\frac{1}{|\oS|}\xi\right),\nabla f + \frac{1}{|\oS|}\xi \right),
\end{equation}
where $f \in \real^{\oS}$. 
\begin{theorem}[Variational Formulation]
Fix $\xi \in \real^d$. Then $f^*$ minimizes the energy $E_{\xi}$,
\begin{equation*}
E_{\xi}(f^*) = \min_{f \in \real^{\oS}} E_{\xi}(f),
\end{equation*}
if and only if $f^*$ is a solution to the modified unit-cell problem~(\ref{eqn:unitcell2}) and thus $f^* = \Upsilon_{\xi}$.
\end{theorem}
\begin{proof}
We note that the Hessian of $E_{\xi}$ is positive semi-definite and thus it is
sufficient to consider the first
order optimality condition. 
For $f,g \in \real^{\oS}$, the derivative of $E_{\xi}$ is given by
\begin{equation*}
\begin{aligned}
DE_{\xi}(f)(g)
&= 	\frac{1}{2}\Big(\nabla g,A(\nabla f + \frac{1}{|\oS|}\xi)\Big) + \frac{1}{2}\Big(\nabla f + \frac{1}{|\oS|}\xi,A\nabla g\Big) \\
&= \Big(\nabla g,A(\nabla f + \frac{1}{|\oS|}\xi)\Big),
\end{aligned}
\end{equation*}
where the last equality follows from Lemma \ref{lem:Asymm}. Then
$DE_{\xi}(f^*) = 0$ if and only if 
\begin{equation*}
\Big(\nabla g, A(\nabla f^* + \frac{1}{|\oS|} \xi)\Big) = 0 \quad
\forall g \in \real^{\oS}.
\end{equation*}
By Lemma~\ref{lem:divergence}, this is equivalent to
\begin{equation*}
\Big(g, \text{div} \big( A(\nabla f^* + \frac{1}{|\oS|} \xi)\big) \Big)
= 0 \quad \forall g \in \real^{\oS}.
\end{equation*}
That is, $f^*$ minimizes $E_{\xi}$ if and only if
\begin{equation*}
\text{div} \big( A(\nabla f^* + \frac{1}{|\oS|} \xi)\big) = 0.
\end{equation*}
By Theorem~\ref{t:unitcell}, the desired result follows. 

\end{proof}

\subsection{Effective diffusivity matrix $K$ in terms of gradient and divergence}
 
We first state a useful lemma before providing a characterization of the  
effective diffusivity $K$ in terms of gradient and divergence. 

\begin{lemma}\label{lem:prediv}
Let $f \in \real^{\oE}$ and $g \in \real^{\oS}$. Suppose $f(x,y) = -f(y,x)$ for all $(x,y) \in \oE$. Then
\begin{equation*}
\sum_{e \in \oE} f(e)(g(\partial_+e) - g(\partial_-e))
= \sum_{e \in \oE} - 2f(e)g(\partial_-e)
\end{equation*}
\end{lemma}

\begin{proof}
\begin{align*}
\sum_{e \in \oE} f(e)(g(\partial_+e) - g(\partial_-e))
= \sum_{e \in \oE} f(e)g(\partial_+e)
	-\sum_{e \in \oE} f(e)g(\partial_-e) \\
= \sum_{y \in \oS }\sum_{e \in \E_y'} f(e)g(\partial_+e)
	-\sum_{e \in \oE} f(e)g(\partial_-e)
= \sum_{e \in \oE} - 2f(e)g(\partial_-e).
\end{align*}
\end{proof}

\begin{lemma}\label{lem:K-rep}
Define $K$ as in (\ref{eqn:K}). Then
\begin{equation*}
K\xi = \sum_{y\in\oS} \sum_{e \in \sE_y} A(e)(\nabla\Upsilon_{\xi}(e) + \frac{1}{|\oS|}\xi), \quad \forall \xi \in \real^d.
\end{equation*}
\end{lemma}
\begin{proof}
Note that from Lemma \ref{lem:prediv}
\begin{align*}
&\sum_{y\in\oS} \sum_{e \in \sE_y} -\nu_{e} \lambda_{e}\Upsilon_{\xi}(y)
= \sum_{y\in\oS} \sum_{e \in \sE_y} -\nu_{e}\nu_e^T \lambda_{e}\Upsilon_{\xi}(y) \frac{\nu_e}{|\nu_e|^2} \\
= &\frac{1}{2} \sum_{y\in\oS} \sum_{e \in \sE_y} \nu_e \nu_e^T\lambda_e\big(\Upsilon_{\xi}(\partial_+e)
		-\Upsilon_{\xi}(\partial_-e)\big)\frac{\nu_e}{|\nu_e|^2}
= \sum_{y\in\oS} \sum_{e \in \sE_y} A(e)\nabla\Upsilon_{\xi}(e).
\end{align*}
Also, $\pi(y) = 1/|\oS|$ for all $y \in \oS$ by symmetry of $L$. Then, from
\eqref{eqn:K}, 
\begin{align*}
K \xi = \frac{1}{2}\sum_{y\in\oS}\sum_{e \in \sE_y} \Big(\frac{\nu_e \nu_e^T }{|\oS|}\lambda_e - 2\nu_e \omega(y)^T \lambda_e\Big)\xi
= \sum_{y\in\oS}\sum_{e \in \sE_y} \nu_e \lambda_e\Big(\frac{\nu_e^T\xi}{2|\oS|} - \omega(y)^T\xi\Big) \\
= \sum_{y\in\oS}\sum_{e \in \sE_y} \nu_e\lambda_e\Big(\frac{\nu_e^T\xi}{2|\oS|} - \Upsilon_{\xi}(y)\Big)
= \sum_{y\in\oS} \sum_{e \in \sE_y}\left(-\nu_{e} \lambda_{e}\Upsilon_{\xi}(y) +
\frac{1}{2 |\oS|}\nu_e \nu_e^T \lambda_e\xi\right) \\
= \sum_{y\in\oS} \sum_{e \in \sE_y} A(e)(\nabla\Upsilon_{\xi}(e) + \frac{1}{|\oS|}\xi).
\end{align*}
\end{proof}

%% file: inputs/numerical-experiments.tex
\section{Numerical experiments}\label{sect:numerics}
Our numerical computations of effective diffusivities described 
in this section are carried out by solving for $\pi$ in
\eqref{eqn:stat-dist}, solving a set of $d$ modified unit-cell problems
\eqref{eqn:unitcell2} with $\xi$ taking the values of the standard basis 
elements of $\real^d$, and then synthesizing $\omega$, the solution to the 
unit-cell problem \eqref{eqn:unitcell}. Once $\omega$ is computed, 
\eqref{eqn:K} is used to compute the effective diffusivity. In this paper, we do not 
investigate the details of the numerical fidelity of these procedures.  
Some details on the numerical methods used may be found in
\cite{preston2017homogenization}. 
  
Our main goal in this section will be to illustrate the effects of path
length,  approaches to modeling attraction, repulsion and bonding, as well as 
comparing our spatially discretized graph model to a continuous space and 
nonzero path length model. While it is possible to apply our method 
to obstructions with more elaborate shapes, for sake of brevity and in keeping 
with the above goal, we shall keep the same simple obstruction geometry
throughout this section.   

We fix a periodic obstruction geometry in two dimensions where a quarter of the periodic cell is obstructed. More precisely, the plane is periodically tiled by a unit square whose upper right quadrant is obstructed. Thus the obstructed region is given by
\begin{equation*}
  \mathcal{O} = \bigcup_{z \in \integ^2} \{z + [3/4,1] \times [3/4,1]\}.
\end{equation*}
Given this geometry, consider a randomly walking solute that undergoes specular reflections at the obstruction boundaries $\partial \mathcal{O}$. In the following subsections we explore different aspects of this random walk with the geometry fixed.

Throughout the following sections, let $(\sS_h,\sE_h,\lambda_h)$ denote a directed, weighted graph where $h > 0$ is a scaling parameter and
\begin{equation}\label{eqn:graph_sq_obs}
\begin{aligned}
	\sS_h &= h\integ^2 \backslash \mathcal{O}, \\
	\sE_h &= \{(x,y) \in \sS_h \times \sS_h \;|\; x - y = \pm(h,0) \text{ or } x - y = \pm(0,h)\}.
\end{aligned}
\end{equation}
That is, the node set excludes the obstructed region $\mathcal{O}$ and edges exist between each node and its nearest neighbors. The jump rate function $\lambda_h$ varies depending on the effects under investigation. Figure~\ref{fig:graph-pd-cell} depicts the quotiented node and edge sets of $\sS_h$ and $\sE_h$ for $h = 1/2, 1/4, 1/8$, and $1/16$.

\subsection{Nonzero path length effects}\label{sect:path-length-effects}
If the mean path length of the random walk is negligible relative to the spacing between obstructions, then one may approximate the random walk by a Brownian motion with reflections. In this case, the probability density function satisfies the diffusion equation with no flux boundary conditions on the obstruction boundaries. This will justify the use of PDE homogenization theory. However, if the mean path length is not sufficiently smaller than the obstruction spacing, the nonzero path length effects need to be investigated. We consider this case by examining a random walk on a subset of a lattice where the jump size is comparable to the obstruction spacing (see Figure~\ref{fig:graph-pd-cell}). We calculate the effective diffusivity coefficients for the graph in~\eqref{eqn:graph_sq_obs} with successively decreasing path (edge) lengths. Let
\begin{equation*}
	\lambda_h(e) = 1/h^2 \text{ for all } e \in \sE_h.
\end{equation*}
The obstructed area remains constant, but the path length is repeatedly
halved. The results in Figure~\ref{fig:pathlengtheffects} show that $D_e$
varies significantly for different $h$ and thus the choice of path length is
meaningful. Additionally, as $h \to 0+$, the effective diffusivity of the graph random walk
model appears to
converge to the effective diffusivity predicted by PDE homogenization. This
suggests a ``backdoor'' approach wherein one might approximate PDE homogenization via homogenizing a graph with a very fine mesh or path length.

\subsection{Modeling an interaction}\label{sect:interaction-effects}
This section explores the impact of an interaction (e.g., attraction, repulsion, or bonding) between the particle and obstruction boundaries on the effective diffusivity. We fix $h = 1/8$ and suppress dependence on $h$. The lower left graph in Figure~\ref{fig:pathlengtheffects} depicts a periodic cell of the graph of interest, $(\sS,\sE,\lambda)$. We choose various rate functions to model four different types of interactions between the random walker and the square obstruction. In the ``Neutral'' case, we simply set $\lambda(e) = 1/h^2$ for all $e \in \sE$.

Define the set of nodes bordering the obstructed region,
\begin{equation}\label{eqn:border-nodes}
\mathcal{B}_h = \{x \in \sS_h \;|\; x + y \in \mathcal{O} \text{ for some } ||y||_{\infty} < h\}.
\end{equation}
In the ``Bonding'' case, the jump rate function is given by
\begin{equation*}
	\lambda(x,y) =
	\begin{cases}
		 1/(2 h^2) & x \in \mathcal{B}_h \\
		 1/h^2 & \text{ otherwise.}
	\end{cases}
\end{equation*}
This rate function slows the random walker whenever it is near the boundary of an obstruction. In the ``Repulsion'' case,
\begin{equation*}
	\lambda(x,y) =
	\begin{cases}
		 2 / h^2 & x \in \mathcal{B}_h, y \notin \mathcal{B}_h \\
		 1/(2 h^2) & x \notin \mathcal{B}_h, y \in \mathcal{B}_h \\
		 1/h^2 & \text{ otherwise,}
	\end{cases}
\end{equation*}
which causes the random walker to be pushed away from the obstructions. In the ``Attraction'' case,
\begin{equation*}
	\lambda(x,y) =
	\begin{cases}
		 2 / h^2 & x \notin \mathcal{B}_h, y \in \mathcal{B}_h \\
		 1/(2 h^2) & x \in \mathcal{B}_h, y \notin \mathcal{B}_h \\
		 1 / h^2 & \text{ otherwise.}
	\end{cases}
\end{equation*}
This rate function pulls the random walker towards the obstructions.

The latter three rate functions illustrate how one might account for interactions between a particle and the obstructions. The effective diffusivity coefficients for each regime are shown in Figure~\ref{fig:drift-effects}. The Bonding rate function yielded the most significant reduction in diffusion. Notably, diffusion in the Attraction regime was faster than the Repulsion regime.
\begin{figure}
\centering
	\includegraphics[width=60mm]{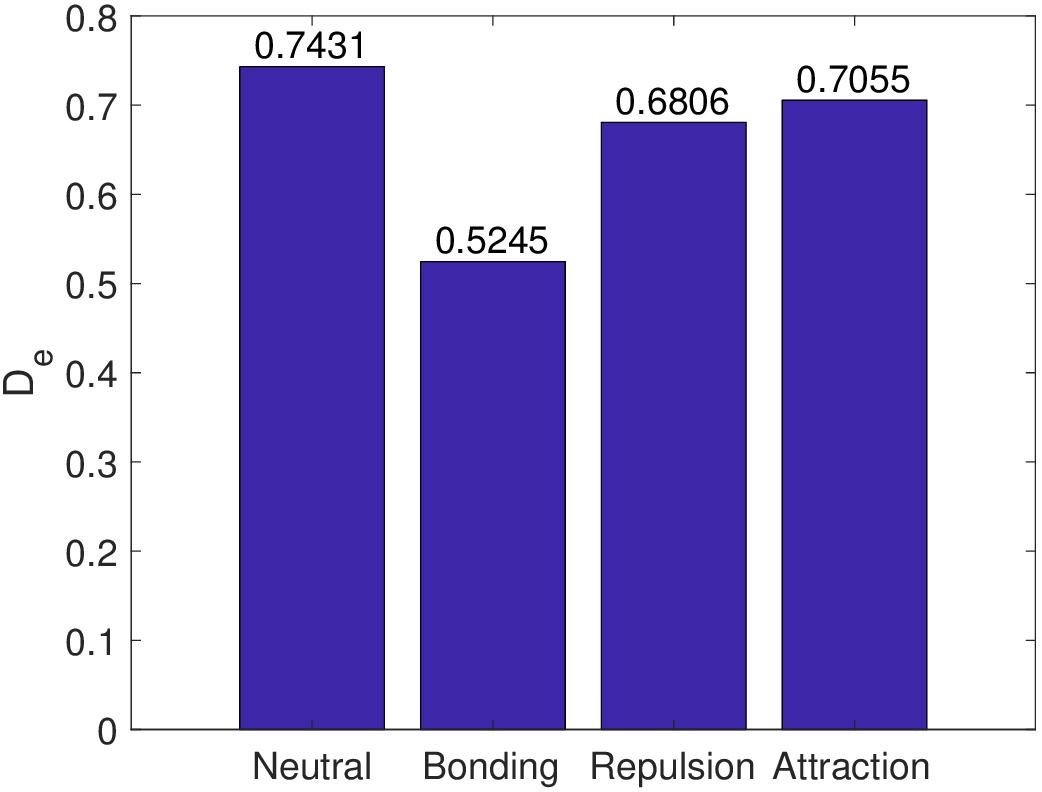}
	\caption{Effective diffusivity coefficients for graphs with nodes and edges given by~\eqref{eqn:graph_sq_obs} and various rate functions. The rates are chosen to model four settings: 1) no interaction, 2) a bonding/sticking effect near the obstruction, 3) an obstruction that repels the random walker, and 4) an obstruction that attracts the random walker.}
\label{fig:drift-effects}
\end{figure}

\subsection{Approximating continuous space}\label{sect:cont-space-comp}
Note that our graph framework confines the particle to discrete spatial locations, which may not reflect physical reality. A more accurate model would consider a random walk with finite path length and \emph{continuous state space} consisting of all points in the unobstructed region. The current graph model does not accommodate continuous space and thus we compare (via Monte Carlo simulations) the continuous space model's predictions with those of the graph model.

Define the process $X_h(t) \in \real^2 \backslash \mathcal{O}$ where $X_h$ waits at a location for time $h^2/4$ and then moves along a straight path of length $h$ in a uniformly chosen random direction. If the path intersects the obstructed region $\mathcal{O}$, then $X$ undergoes a specular (mirror-like) reflection to ensure its path does not intersect $\mathcal{O}$. Because $X_h(t)$ exhibits normal diffusion, the effective diffusivity and mean squared displacement $MSD(t)$ are related by
\begin{equation*}
	D_e = \lim_{t \rightarrow \infty}\frac{MSD(t)}{4t}.
\end{equation*}
We use Monte Carlo simulation to compute an estimate $\widetilde{MSD}(t)$ of $MSD(t)$ at equally spaced points in time $\{t_i\}_{i = 1}^N$. Let $\alpha$ denote the slope of the line of best fit (with 0 intercept) through the points $\big\{\big(t_i,\widetilde{MSD}(t_i)\big)\big\}_{i = 1}^N$. We then estimate $D_e$ as
\begin{equation*}
	D_e \approx \alpha/4
\end{equation*}
and calculate the 95\% confidence intervals appropriately.

We can approximate $X_h$ via the graph $(\sS_h,\sE_h,\lambda_h)$ where $\lambda_h(e) = 1/h^2$. We also introduce a slightly more sophisticated graph $(\widetilde{\sS}_h,\widetilde{\sE}_h,\widetilde{\lambda}_h)$ that accounts for jumps to diagonal neighbors and thus ought to better approximate $X_h$. Let
\begin{equation*}
\begin{aligned}
	\widetilde{\sS}_h &= \sS_h, \\
	\widetilde{\sE}_h &= \sE_h \cup \{(x,y) \in \widetilde{\sS}_h \times \widetilde{\sS}_h \;|\; x - y = \pm(h,h)\}.
\end{aligned}
\end{equation*}
The rate function $\widetilde{\lambda}_h$ should replicate the dynamics of $X_h$ reflecting off of an obstruction. Specifically, when the solute has an approximately diagonal collision with an obstruction, the solute will experience a non-negligible displacement (rather than returning to its original location). We therefore double the jump rates along edges that begin and end on the obstruction boundaries $\mathcal{B}_h$~\eqref{eqn:border-nodes}. Thus we define
\begin{equation*}
\widetilde{\lambda}_h(x,y) =
\begin{cases}
	1/h^2 & x,y \in \mathcal{B}_h \\
  1/(2h^2) & \text{otherwise.}
\end{cases}
\end{equation*}

The effective diffusivity coefficients (for $h = 2^{-2},2^{-3},\dots,2^{-8}$) of the continuous process $X_h$, the graph $(\sS_h,\sE_h,\lambda_h)$, and the graph $(\widetilde{\sS}_h,\widetilde{\sE}_h,\widetilde{\lambda}_h)$ are shown in Figure~\ref{fig:cont-disc-discdiag}. Incorporating diagonal jumps and thoughtful jump rates in $(\widetilde{\sS}_h,\widetilde{\sE}_h,\widetilde{\lambda}_h)$ yields a superior approximation of the continuous setting. In all three settings, $D_e$ appears to converge to the same value as $h \rightarrow 0$.
\begin{figure}
\centering
	\includegraphics[width=60mm]{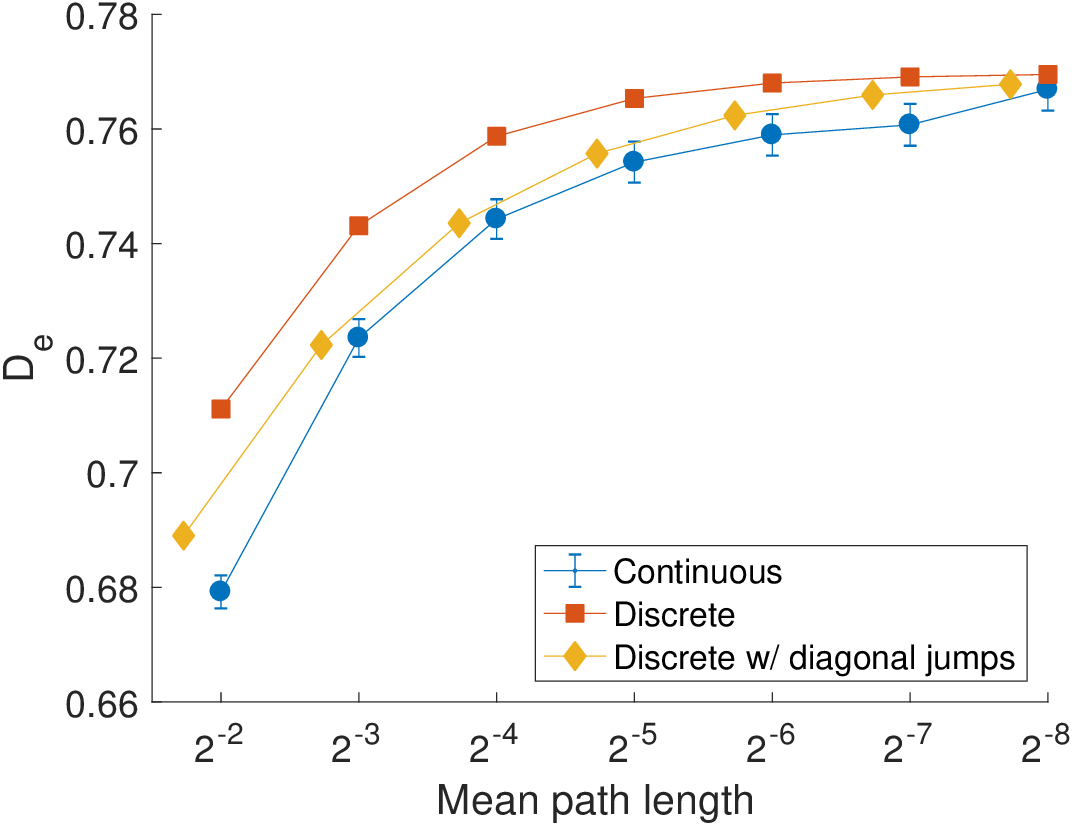}
	\caption{Effective diffusivity coefficients for the continuous process $X_h$ (``Continuous''), the graph $(\sS_h,\sE_h,\lambda_h)$ (``Discrete''), and the graph $(\widetilde{\sS}_h,\widetilde{\sE}_h,\widetilde{\lambda}_h)$ (``Discrete w/ diagonal jumps'') for $h = 2^{-2},2^{-3},\dots,2^{-8}$.}
\label{fig:cont-disc-discdiag}
\end{figure}

%% file: inputs/appendix.tex
\appendix
\section{Validity of the formal result}
In this section we show that the diffusivity derived by the formal
asymptotics is correct. That is, $K = C$ (see (\ref{eqn:K}), (\ref{eqn:C}),
and Theorem~\ref{thm:deff_rigorous}). Noting that the diffusivity 
matrix $K$ was derived under the null-drift condition \eqref{eq:null-drift}, we proceed by 
assuming it. Thus $\tilde{\rho}=\rho$. First, we show two auxiliary results.
\begin{lemma}\label{lem:omega-g}
Let $\omega$ be a solution to the unit-cell problem~(\ref{eqn:unitcell}) and $\psi$ satisfy (\ref{eqn:Lg_fd}). Then
\begin{equation*}
\begin{aligned}
\sum_{y \in \oS}\sum_{e \in \E_y} \nu_e \omega(y)^T \lambda_e
&= \sum_{y \in \oS} \sum_{e \in \E_y} \psi(\partial_+e) \nu_e^T \lambda_e \pi(y),\\
\sum_{y \in \oS}\sum_{e \in \E_y} \omega(y) \nu_e^T \lambda_e
&= \sum_{y \in \oS} \sum_{e \in \E_y} \nu_e \psi(\partial_+e)^T \lambda_e \pi(y).
\end{aligned}
\end{equation*}
\end{lemma}
\begin{proof}
We note that the second equation is simply the transpose of the first.
The result follows easily:
\begin{equation*}
\begin{aligned}
\sum_{y \in \oS}\sum_{e \in \E_y} \nu_e^i \omega_j(y) \lambda_e
= \big( \rho_i, \omega_j \big)
= ( L\psi_i, \omega_j )
= ( \psi_i, L^T\omega_j )
= ( \psi_i, \sigma_j ) \\
= \sum_{y \in \oS} \psi_i(y) \sum_{e \in \E'_y} \nu_e^j \pi(\partial_-e) \lambda_e
= \sum_{y \in \oS} \sum_{e \in \E'_y} \psi_i(\partial_+e) \nu_e^j \pi(\partial_-e) \lambda_e \\
= \sum_{y \in \oS} \sum_{e \in \E_y} \psi_i(\partial_+e) \nu_e^j \lambda_e \pi(y).
\end{aligned}
\end{equation*}
Here, we used superscripts and subscripts as convenient to denote
a component of a vector valued function.
\end{proof}
\begin{lemma}\label{lem:f-g-relationship}
The drift field $\rho$ and $\psi$ (see (\ref{eqn:Lg_fd})) satisfy:
\begin{equation}
\begin{aligned}
&\sum_{y \in \oS} \sum_{e \in \sE_y} \nu_e \psi(\partial_-e)^T \lambda_e \pi(y)
= \sum_{y \in \oS} \rho(y) \psi(y)^T \pi(y),\\
&\sum_{y \in \oS} \sum_{e \in \sE_y} \psi(\partial_-e) \nu_e^T \lambda_e \pi(y)
= \sum_{y \in \oS} \psi(y) \rho(y)^T \pi(y),\\
&\sum_{y \in \oS} \sum_{e \in \sE_y}
(\psi(\partial_+e)-\psi(\partial_-e)) \psi(\partial_-e))^T\lambda_e \pi(y) =\sum_{y \in \oS} \psi(y) \rho(y)^T \pi(y),\\
&\sum_{y \in \oS} \sum_{e \in \sE_y} (\psi(\partial_+e)-\psi(\partial_-e)) (\psi(\partial_+e) - \psi(\partial_-e))^T\lambda_e \pi(y)\\
&\quad = -\sum_{y \in \oS} \psi(y) \rho(y)^T \pi(y) - \sum_{y \in \oS} \rho(y) \psi(y)^T \pi(y).
\end{aligned}
\end{equation}
\end{lemma}
\begin{proof}
The first two equalities which are transposes of each other follow easily
using the definition of $\rho$. The third equality also follows easily using
the relationship $L \psi = \rho$.
To show the fourth equality,  notice that
\begin{equation*}
\begin{aligned}
&  \sum_{y \in \oS} \sum_{e \in \sE_y} (\psi(\partial_+e)-\psi(\partial_-e)) \psi(\partial_+e)^T\lambda_e\pi(y)
	+ \sum_{y \in \oS} \psi(y) \rho(y)^T \pi(y)  \\
=&\; \sum_{y \in \oS} \sum_{e \in \sE_y} (\psi(\partial_+e)-\psi(\partial_-e)) \psi(\partial_+e)^T \lambda_e\pi(y)
	+ \sum_{y \in \oS} \psi(y) (L\psi)(y)^T \pi(y)  \\
=&\; \sum_{y \in \oS} \sum_{e \in \sE_y} \psi(\partial_+e)\psi(\partial_+e)^T \lambda_e\pi(\partial_-e)
	- \sum_{y \in \oS} \sum_{e \in \sE_y}\psi(\partial_-e)\psi(\partial_-e)^T \lambda_e\pi(\partial_-e) \\
=&\; \sum_{y \in \oS} \sum_{e \in \sE_y'} \psi(\partial_+e)\psi(\partial_+e)^T \lambda_e\pi(\partial_-e)
	- \sum_{y \in \oS} \psi(y)\psi(y)^T\pi(y)\sum_{e \in \sE_y} \lambda_e \\
=&\; \sum_{y \in \oS} \psi(y)\psi(y)^T \Big(\sum_{e \in \sE_y'} \pi(y-\nu_e)\lambda_e
- \pi(y)\lambda^0(y)\Big) = \sum_{y \in \oS} \psi(y)\psi(y)^T (L^T \pi)(y)
	= 0,
	\end{aligned}
\end{equation*}
where we have used the relation $L \psi = \rho$, skipped some steps of the
algebra and used \eqref{eqn:sum-of-sums}.
Combine this equality with the third equality to obtain the fourth equality.
\end{proof}

\begin{theorem}
The effective diffusivities derived from the formal and rigorous derivation are equal. That is, $K = C$ where $K$ and $C$ are defined in~(\ref{eqn:K}) and~(\ref{eqn:C}), respectively.
\end{theorem}
\begin{proof}
Expanding the definition of $C$~(\ref{eqn:C}) using \eqref{eqn:alpha}
and comparing with the definition of $K$~\eqref{eqn:K} we see that
$C=K$ if and only if
\begin{equation*}
\begin{aligned}
 &\sum_{y\in\oS} \sum_{e \in \sE_y} \nu_e \omega(y)^T\lambda_e + \omega(y)\nu_e^T\lambda_e \\
=&\;\sum_{y \in \oS} \sum_{e \in \sE_y} \nu_e \big(\psi(\partial_+e) - \psi(\partial_-e)\big)^T \lambda_e \pi(\partial_-e) \\
	&+ \sum_{y \in \oS} \sum_{e \in \sE_y} \big(\psi(\partial_+e) - \psi(\partial_-e)\big) \nu_e^T \lambda_e \pi(\partial_-e) \\
	&-\sum_{y \in \oS} \sum_{e \in \sE_y} (\psi(\partial_+e)-\psi(\partial_-e))(\psi(\partial_+e)-\psi(\partial_-e))^T\lambda_e\pi(\partial_-e).
\end{aligned}
\end{equation*}
It is not difficult to verify this equality using Lemmas~\ref{lem:omega-g} and~\ref{lem:f-g-relationship}.
\end{proof}

%% file: article.bbl
\begin{thebibliography}{10}

\bibitem{ball2006asymptotic}
{\sc K.~Ball, T.~G. Kurtz, L.~Popovic, G.~Rempala, et~al.}, {\em Asymptotic
  analysis of multiscale approximations to reaction networks}, The Annals of
  Applied Probability, 16 (2006), pp.~1925--1961.

\bibitem{bensoussan2011asymptotic}
{\sc A.~Bensoussan, J.-L. Lions, and G.~Papanicolaou}, {\em Asymptotic analysis
  for periodic structures}, vol.~374, American Mathematical Soc., 2011.

\bibitem{berger2007quenched}
{\sc N.~Berger and M.~Biskup}, {\em Quenched invariance principle for simple
  random walk on percolation clusters}, Probability theory and related fields,
  137 (2007), pp.~83--120.

\bibitem{bhattacharya1982functional}
{\sc R.~N. Bhattacharya}, {\em On the functional central limit theorem and the
  law of the iterated logarithm for markov processes}, Zeitschrift f{\"u}r
  Wahrscheinlichkeitstheorie und verwandte Gebiete, 60 (1982), pp.~185--201.

\bibitem{billingsley2011convergence}
{\sc P.~Billingsley}, {\em Convergence of probability measures}, Wiley, 2011.

\bibitem{biskup2011recent}
{\sc M.~Biskup}, {\em Recent progress on the random conductance model},
  Probability Surveys, 8 (2011), pp.~294--373,
  \url{https://doi.org/10.1214/11-ps190}.

\bibitem{bremaud2013markov}
{\sc P.~Br{\'e}maud}, {\em Markov chains: Gibbs fields, Monte Carlo simulation,
  and queues}, vol.~31, Springer Science \& Business Media, 2013.

\bibitem{caputo2003finite}
{\sc P.~Caputo and D.~Ioffe}, {\em Finite volume approximation of the effective
  diffusion matrix: the case of independent bond disorder}, in Annales de
  l'Institut Henri Poincare (B) Probability and Statistics, vol.~39, Elsevier,
  2003, pp.~505--525.

\bibitem{desbrun2005discrete}
{\sc M.~Desbrun, A.~N. Hirani, M.~Leok, and J.~E. Marsden}, {\em Discrete
  exterior calculus}, arXiv preprint math/0508341,  (2005).

\bibitem{preston2017homogenization}
{\sc P.~Donovan}, {\em Homogenization theory for solute motion accounting for
  obstructions, interactions, and path length effects}, PhD thesis, University
  of Maryland Baltimore County, 2017.

\bibitem{donovan2016homogenization}
{\sc P.~Donovan, Y.~Chehreghanianzabi, M.~Rathinam, and S.~P. Zustiak}, {\em
  Homogenization theory for the prediction of obstructed solute diffusivity in
  macromolecular solutions}, Plos One, 11 (2016),
  \url{https://doi.org/10.1371/journal.pone.0146093}.

\bibitem{egloffe2014random}
{\sc A.-C. Egloffe, A.~Gloria, J.-C. Mourrat, and T.~N. Nguyen}, {\em Random
  walk in random environment, corrector equation and homogenized coefficients:
  from theory to numerics, back and forth}, IMA journal of numerical analysis,
  35 (2014), pp.~499--545.

\bibitem{ethier2005markov}
{\sc S.~N. Ethier and T.~G. Kurtz}, {\em Markov processes characterization and
  convergence}, Wiley-Interscience, 2005.

\bibitem{faggionato2008random}
{\sc A.~Faggionato}, {\em Random walks and exclusion processes among random
  conductances on random infinite clusters: homogenization and hydrodynamic
  limit}, Electronic Journal of Probability, 13 (2008), pp.~2217--2247,
  \url{https://doi.org/10.1214/ejp.v13-591}.

\bibitem{friedman2004calculus}
{\sc J.~Friedman and J.-P. Tillich}, {\em Calculus on graphs}, arXiv preprint
  cs/0408028,  (2004).

\bibitem{gloria2016quantitative}
{\sc A.~Gloria and J.~Nolen}, {\em A quantitative central limit theorem for the
  effective conductance on the discrete torus}, Communications on Pure and
  Applied Mathematics, 69 (2016), pp.~2304--2348.

\bibitem{gloria2011optimal}
{\sc A.~Gloria, F.~Otto, et~al.}, {\em An optimal variance estimate in
  stochastic homogenization of discrete elliptic equations}, The annals of
  probability, 39 (2011), pp.~779--856.

\bibitem{grady2010discrete}
{\sc L.~J. Grady and J.~R. Polimeni}, {\em Discrete calculus: Applied analysis
  on graphs for computational science}, Springer Science \& Business Media,
  2010.

\bibitem{hirani2003discrete}
{\sc A.~N. Hirani}, {\em Discrete exterior calculus}, PhD thesis, California
  Institute of Technology, 2003.

\bibitem{kang2013separation}
{\sc H.-W. Kang, T.~G. Kurtz, et~al.}, {\em Separation of time-scales and model
  reduction for stochastic reaction networks}, The Annals of Applied
  Probability, 23 (2013), pp.~529--583.

\bibitem{kang2014central}
{\sc H.-W. Kang, T.~G. Kurtz, L.~Popovic, et~al.}, {\em Central limit theorems
  and diffusion approximations for multiscale markov chain models}, The Annals
  of Applied Probability, 24 (2014), pp.~721--759.

\bibitem{keller1977effective}
{\sc J.~B. Keller}, {\em Effective behavior of heterogeneous media}, in
  Statistical mechanics and statistical methods in theory and application,
  Springer, 1977, pp.~631--644.

\bibitem{kipnis1986central}
{\sc C.~Kipnis and S.~S. Varadhan}, {\em Central limit theorem for additive
  functionals of reversible markov processes and applications to simple
  exclusions}, Communications in Mathematical Physics, 104 (1986), pp.~1--19.

\bibitem{klebaner2012introduction}
{\sc F.~C. Klebaner}, {\em Introduction to stochastic calculus with
  applications}, Imperial College Press, 2012.

\bibitem{kozlov1987averaging}
{\sc S.~Kozlov}, {\em Averaging of difference schemes}, Mathematics of the
  USSR-Sbornik, 57 (1987), p.~351.

\bibitem{kozlov1980homogenization}
{\sc S.~Kozlov, O.~Olenik, and V.~Zhikov}, {\em Homogenization of differential
  operators}, Russian Mathem. Surveys, 35 (1980), p.~254.

\bibitem{kunnemann1983diffusion}
{\sc R.~K{\"u}nnemann}, {\em The diffusion limit for reversible jump processes
  onz d with ergodic random bond conductivities}, Communications in
  Mathematical Physics, 90 (1983), pp.~27--68.

\bibitem{mathieu2008quenched}
{\sc P.~Mathieu}, {\em Quenched invariance principles for random walks with
  random conductances}, Journal of Statistical Physics, 130 (2008),
  pp.~1025--1046.

\bibitem{owhadi2003approximation}
{\sc H.~Owhadi}, {\em Approximation of the effective conductivity of ergodic
  media by periodization}, Probability theory and related fields, 125 (2003),
  pp.~225--258.

\bibitem{zhou2005learning}
{\sc D.~Zhou, J.~Huang, and B.~Sch{\"o}lkopf}, {\em Learning from labeled and
  unlabeled data on a directed graph}, in Proceedings of the 22nd international
  conference on Machine learning, ACM, 2005, pp.~1036--1043.

\end{thebibliography}
